\documentclass[11pt]{amsart}
\usepackage{stmaryrd}
\usepackage{amsmath}
\usepackage{fullpage}
\usepackage{xspace}
\usepackage[psamsfonts]{amssymb}
\usepackage[latin1]{inputenc}
\usepackage{graphicx,color}
\usepackage{hyperref}
\usepackage{graphicx}
\usepackage{parskip}
\usepackage[
    backend=biber, 
    style=alphabetic, 
    maxnames=99,
]{biblatex}
\usepackage{amsmath}%
\usepackage{amsthm}%
\usepackage{amscd}
\usepackage{amsfonts}%
\usepackage{amssymb}%
\usepackage{graphicx}
\usepackage{stmaryrd}
\usepackage{mathrsfs}
\usepackage{tikz}
\usetikzlibrary{matrix,arrows}
\usepackage{tikz-cd}
\usepackage[all]{xy}
\usepackage{bbm} 
\usepackage{longtable}

\newtheorem{theorem}{Theorem}[section]

\newtheorem{corollary}[theorem]{Corollary}

\newtheorem{definition}{Definition}

\newtheorem{lemma}[theorem]{Lemma}

\newtheorem{proposition}[theorem]{Proposition}
\theoremstyle{remark}
\newtheorem{remark}[theorem]{Remark}

\newenvironment{customtheorem}[1]{%
  \renewcommand\thetheorem{#1}
  \theorem
}{\endtheorem}

\numberwithin{equation}{section}

\def\i{\boldsymbol{\mathrm{i}}}
\def\v{\boldsymbol{{v}}}
\def\e{\boldsymbol{{e}}}
\def\r{\boldsymbol{\mathrm{r}}}

\newcommand{\afrak}{\mathfrak{a}}

\newcommand{\ffrak}{\mathfrak{f}}

\newcommand{\mfrak}{\mathfrak{m}}
\newcommand{\nfrak}{\mathfrak{n}}

\newcommand{\pfrak}{\mathfrak{p}}

\newcommand{\Ocal}{\mathscr{O}}

\newcommand{\Tcal}{\mathscr{T}}

\newcommand{\C}{\mathbb{C}}

\newcommand{\F}{\mathbb{F}}

\newcommand{\I}{\mathbb{I}}

\newcommand{\K}{\mathbb{K}}

\newcommand{\N}{\mathbb{N}}

\newcommand{\R}{\mathbb{R}}

\newcommand{\T}{\mathbb{T}}
\newcommand{\U}{\mathbb{U}}

\newcommand{\Z}{\mathbb{Z}}

\newcommand{\Ocali}{\mathcal{O}}

\newcommand{\Tcali}{\mathcal{T}}

\newcommand{\GL}{\mathrm{GL}}
\newcommand{\PGL}{\mathrm{PGL}}

\newcommand\rquot[2]{
  \mathchoice
  {
    \text{\raise0.5ex\hbox{$#1$}\big/\lower0.5ex\hbox{$#2$}}%
  }
  {
    #1\,/\,#2
  }
  {
    #1\,/\,#2
  }
  {
    #1\,/\,#2
  }
}

\newcommand\lrquot[3]{
  \mathchoice
  {
    \text{\lower0.5ex\hbox{$#1$}\big\backslash\raise0.5ex\hbox{$#2$\!}\big/
      \lower0.5ex\hbox{\!\!$#3$}}%
  }
  {
    #1\,\backslash\,#2\,/\,#3
  }
  {
    #1\,\backslash\,#2\,/\,#3
  }
  {
    #1\,\backslash\,#2\,/\,#3
  }
}

\newcommand\lquot[2]{
  \mathchoice
  {
    \text{\lower0.5ex\hbox{$#1$}\big\backslash\raise0.5ex\hbox{$#2$}}%
  }
  {
    #1\,\backslash\,#2
  }
  {
    #1\,\backslash\,#2
  }
  {
    #1\,\backslash\,#2
  }
}

\usepackage{fontawesome}

  \DeclareFontFamily{U}{wncy}{}
    \DeclareFontShape{U}{wncy}{m}{n}{<->wncyr10}{}
    \DeclareSymbolFont{mcy}{U}{wncy}{m}{n}
    \DeclareMathSymbol{\Sha}{\mathord}{mcy}{"58}

\begin{document}
\title{Asymptotic distribution of CM points on the reduction of the Drinfeld modular curve}
\author{Matias Alvarado and Patricio P\'erez-Pi\~na}
\address{ Instituto de matem\'aticas, Universidad de Talca, Talca, Chile}
\email[M. Alvarado]{matias.alvarado@utalca.cl }
\address{ Department of Mathematical Sciences, University of Copenhagen}
\email[P. P\'erez-Pi\~na]{papp@math.ku.dk }%
\date{\today}

\begin{abstract}

We study a distribution problem over global function 
fields. More precisely, we describe the asymptotic 
distribution of rank $2$ CM Drinfeld modules among the 
irreducible components of the analytic reduction of the 
Drinfeld modular curve.
Our approach relies on the properties of the building map and the spectral decomposition of the adjacency operator on a 
quotient of the Bruhat-Tits tree.
\end{abstract}
\maketitle
\setcounter{tocdepth}{1}


\section{Introduction}


In this article, we present a distribution result for CM Drinfeld modules of rank 2. More concretely, we describe the asymptotic proportion of those reducing to a fixed irreducible component of the rigid-analytic reduction of the Drinfeld modular curve. This result resembles distribution problems related with the supersingular reduction of CM points. To better ilustrate this, let us start by reviewing these type of results in the more classical case of elliptic curves.


Recall that an elliptic curve $E$ over $\C$ has complex multiplication if its endomorphism ring is an order $\Ocal$ in an imaginary quadratic field $K$. The discriminant of $E$ is the negative integer $d$, defined as the discriminant of $\Ocal$. We can always write $d$ in the form $d=d_Kc^{2}$, where $d_K$ is the discriminant of $K$, and $c\geq1$ an integer. The integer $d_K$ is called the fundamental discriminant, and $c$ is called the conductor. When reducing at a fixed rational prime $\ell$, a CM elliptic curve $E$ has supersingular reduction if and only if $\ell$ does not split in $K$. In this situation, Galois orbits of CM elliptic curves become uniformly distributed in the supersingular locus when $|d|\to\infty$. This was first studied by Cornut and Vatsal (see \cite{Cornut02},\cite{Vatsal02}, and \cite{CornutVatsal}), whose methods apply to the case of fixed fundamental discriminant and varying conductor of the form $c=p^n$, where $p\neq\ell$ is a prime. The article \cite{Michel04} by Michel describes the case where $d=d_K$ varies, and his method guarantees uniform distribution of incomplete Galois orbits. In \cite{JetchevKane}, the authors extend this results to arbitrary variation in $d$. All the aforementioned results are subject to the condition that $\ell$ is inert in $K$. A statement covering all cases, i.e., when $\ell$ is inert or ramified, can be found in \cite[Theorem 5.7]{herreroMenaresRivera2}.

Replacing CM elliptic curves by CM Drinfeld modules of rank $2$ over rational function fields, Theorem 1.3 in \cite{papanikolasequidistribution} is the function field analogue of Michel's result. It proves the uniform distribution of incomplete Galois orbits of CM Drinfeld modules of rank 2 with respect to their reduction in the supersingular locus. The reduction is considered modulo an inert prime, and the discriminants allowed are irreducible of odd degree. 

Both in the case of CM elliptic curves over the complex numbers and CM Drinfeld modules of rank $2$ over rational function fields, their distribution on the supersingular locus can be interpreted as a distribution problem for Gross points among irreducible components of certain definite Shimura curves (see Section 6.1 in \cite{Michel04} and Section 1.3 in \cite{papanikolasequidistribution}, respectively). 
Thus, our main results (Theorem \ref{thmA}, Theorem \ref{thmB} and Theorem \ref{thmC} below), which describe the asymptotic distribution of CM Drinfeld modules of rank 2 among the irreducible components of the rigid-analytic reduction of the Drinfeld modular curve $Y$ (when viewed as a rigid analytic space over $\C_\infty$ via its uniformization by the Drinfeld upper half-plane), can be thought of as a rigid-analytic version of the supersingular reduction of CM points. We allow for arbitrary variation in the sequence of discriminants. We now proceed to explain our results. 

Let $q$ be a power of an odd prime number, and let $A$ be the polynomial ring $\F_q[T]$ with field of fractions $k=\F_q(T)$. The field $k_\infty=\F_q((T^{-1}))$ is the completion of $k$ at the infinite place $\infty=T^{-1}$. Let $\C_\infty$ be the completion of an algebraic closure of $k_\infty$, and denote by $|\cdot|$ the absolute value on $\C_\infty$, normalized by $|T|=q$. Consider the group $\Gamma=\PGL_2(A)$. The Drinfeld upper half-plane is the rigid analytic space \[\Omega=\C_\infty\smallsetminus k_\infty.\] Using the analytic uniformization of Drinfeld modules (see Section \ref{Drinfeldmodules}.), one obtains an identification \[\Gamma\backslash\Omega\cong Y(\C_\infty)\] that sends the $\Gamma$-class of $z\in \Omega$ to the isomorphism class of the Drinfeld $A$-module of rank $2$ corresponding to the quotient $\C_\infty/(A+zA)$. 

The space $Y(\C_\infty)$ comes equipped with a rigid-analytic reduction which we now describe following \cite{reversat} and \cite{gekelerreversat}. The space $\Omega$ admits a rigid-analytic reduction $\pi\colon\Omega\to\overline{\Omega}$, where $\overline{\Omega}$ is a scheme over $\overline{\F_q}$, locally of finite type, and each of its irreducible components is a $\mathbb{P}^1_{\overline{\F_q}}$ which meets exactly $(q+1)$ other irreducible components, each of them in one point which is ordinary and rational over $\F_q$. The group $\Gamma$ also acts on $\overline{\Omega}$ and $\pi$ respects this action. Therefore, we obtain a reduction map \[\mathrm{red}\colon Y(\C_\infty)\cong\Gamma\backslash\Omega\to\Gamma\backslash\overline{\Omega}.\]


Let $\Tcal$ denote the intersection graph of $\overline{\Omega}$. That is, the vertices of $\Tcal$ are the irreducible components of $\overline{\Omega}$ and two vertices are connected by an edge if the corresponding irreducible components meet. There exists a canonical identification between $\Tcal$ and $\Tcali$, the Bruhat--Tits tree of $\PGL_2(k_\infty)$. The intersection graph of the reduction $\Gamma\backslash\overline{\Omega}$ is thus canonically isomorphic to $\Gamma\backslash\Tcali$. This quotient has the shape of an infinite ray

\begin{equation}\label{line}\tikz[baseline=-1.5ex]{
  \def\s{0.7}      
  \def\d{0.18}     

  \foreach \i in {0,...,3} {
    \node[circle,fill,inner sep=1.2pt,label=below:$v_{\i}$] (n\i) at (\s*\i,0) {};
  }

  \foreach \i in {0,...,2} {
    \draw (n\i) -- (n\the\numexpr\i+1\relax);
  }

  \draw (n3) -- ++(\s,0);

  \foreach \j in {1,2,3} {
    \node[circle,fill,inner sep=0.6pt] at (\s*4 + \d*\j, 0) {};
  }
}\end{equation} 

For $n\geq0$, we denote by $C_n$ the irreducible component of $\Gamma\backslash\overline{\Omega}$ associated to the vertex $v_n$ following the notation of \eqref{line}. For example, $C_0$ is the unique irreducible component of $\Gamma\backslash\overline{\Omega}$ that intersects only one other irreducible component named $C_1$. For $n\geq0$, denote by $c_{n+1}$ the intersection $C_{n}\cap C_{n+1}$.

A Drinfeld $A$-module over $\C_\infty$ has complex multiplication if its endomorphism ring is an order in a quadratic extension of $k$ for which the place $\infty$ is not split. Such extensions are called \textit{imaginary quadratic extensions} of $k$. Let $D$ be a non-square in $A$ such that $K_D=k(\sqrt{D})$ is an imaginary quadratic extension of $k$. We denote by $\mathrm{CM}(D)\subseteq Y(\C_\infty)$ the set of all CM Drinfeld modules of rank $2$ with complex multiplication by $\mathcal{O}_D:=A[\sqrt{D}]$. Let $D_K$ be the square-free part of $D$ and write $D=D_Kf^2$ with $f\in A$ is monic. Then $\mathcal{O}_{K_D}=A[\sqrt{D_K}]$ is the maximal $A$-order of $K_D$ and we call $f$ the conductor of $D$. See \cite[Proposition 17.6]{rosen}.


Each element of $\mathrm{CM}(D)$ reduces to a unique irreducible component $C_n$ if and only if $\infty$ is an inert prime in $K_D$. This condition holds if and only if $D$ is of even degree and its leading term is not a square of $\mathbb{F}_q$ (see \cite[Proposition 14.6]{rosen}). We call such discriminants, \textit{inert discriminants}. When $D$ is of odd degree, the extension $K_D/k$ is ramified at $\infty$ and the reduction of elements in $\mathrm{CM}(D)$ lands on the intersections $\{c_{n+1}\}_{n\geq0}$.

In contrast with the reduction of CM points on the supersingular locus previously discussed, the asymptotic distribution of $\mathrm{CM}(D)$ among the irreducible components $(C_n)_{n\geq0}$ depends on the arithmetic relation between $\infty$ and $D$. In this regard, it is slightly similar to the situation of $p$-adic distribution of CM elliptic curves (see \cite{herreroMenaresRivera1} and \cite{herreroMenaresRivera2}). For this reason, we present our results in three theorems, two in the inert case and one covering the ramified case. The first one below deals with the case of a sequence of inert discriminants $D$ for which the sequence of square-free parts $D_K$ is eventually of positive degree. This is equivalent to say that $K_D/k$ is inert and eventually different from the constant field extension $\F_{q^2}(T)$.

\begin{customtheorem}{A}[]\label{thmA}
Consider a sequence of inert discriminants $D$ such that $\deg D_K>0$ if $|D|$ is sufficiently large. Then, for all $\varepsilon>0$ we have \[\left|\frac{\mathrm{CM}(D)\cap\mathrm{red}^{-1}(C_0)}{\#\mathrm{CM}(D)}- \frac{q-1}{2q}\right|\ll_\varepsilon |D|^{-1/4+\varepsilon}\quad\mbox{ as }|D|=q^{\deg D}\to\infty\] and for $n\geq1$, \[\left|\frac{\mathrm{CM}(D)\cap\mathrm{red}^{-1}(C_n)}{\#\mathrm{CM}(D)}- \frac{q^2-1}{2q^{n+1}}\right|\ll_{n,\varepsilon} |D|^{-1/4+\varepsilon}\quad\mbox{ as }|D|=q^{\deg D}\to\infty.\]
\end{customtheorem}

The next result treats the remaining inert case. That is, it considers a sequence of inert discriminants $D$ for which $\deg(D_K)=0$, or equivalently, the quadratic field $K_D=\F_{q^2}(T)$ is fixed.

\begin{customtheorem}{B}[]\label{thmB}
Let $K=\F_{q^2}(T)$ and consider a sequence of discriminants of the form $D=D_{K}f^2$. For $n\geq0$, if $\deg f$ has different parity than $n$, the intersection $\mathrm{red}(\mathrm{CM}(D))\cap C_n$ is empty. Moreover, if $\deg f$ is always even, then for all $\varepsilon>0$ we have \[\left|\frac{\mathrm{CM}(D)\cap\mathrm{red}^{-1}(C_0)}{\#\mathrm{CM}(D)}- \frac{q-1}{q}\right|\ll_\varepsilon |D|^{-1/4+\varepsilon}\quad\mbox{ as }|D|=q^{2\deg f}\to\infty.\] For $n\geq1$, if $\deg f$ has same parity as $n$, then for all $\varepsilon>0$ we have \[\left|\frac{\mathrm{CM}(D)\cap\mathrm{red}^{-1}(C_n)}{\#\mathrm{CM}(D)}- \frac{q^2-1}{q^{n+1}}\right|\ll_{n,\varepsilon} |D|^{-1/4+\varepsilon}\quad\mbox{ as }|D|=q^{2\deg f}\to\infty.\]
\end{customtheorem}

\begin{remark}
    In \cite{AABP24}, the authors consider sets $C_\varepsilon(D)$, for $0<\varepsilon\leq1$, whose cardinality is estimated as a step towards proving the finiteness of Drinfeld singular moduli that are units. The set $\mathrm{CM}(D)\cap\mathrm{red}^{-1}(C_0)$ can be seen to agree with the set $C_1(D)$. 
\end{remark}

Finally, we treat the ramified case.

\begin{customtheorem}{C}[]\label{thmC}
Consider a sequence of odd degree discriminants $D$. For all $n\geq0$ and $\varepsilon>0$, we have \[\left|\frac{\mathrm{CM}(D)\cap\mathrm{red}^{-1}(c_{n+1})}{\#\mathrm{CM}(D)}- \frac{q-1}{q^{n+1}}\right|\ll_\varepsilon |D|^{-1/4+\varepsilon}\quad\mbox{ as }|D|=q^{\deg D}\to\infty\]
\end{customtheorem}

\subsection*{Overview of the article} In contrast to the supersingular locus, the number of irreducible components of $\Gamma\backslash\overline{\Omega}$ is not finite. Thus, our proof does not stem from the methods employed in the case of reduction modulo a finite prime. Instead, we solve this problem using harmonic analysis on the tree, an approach reminiscent of \cite{dukeHiperbolic}. We now briefly explain how this is achieved. 

The goal of Section \ref{S reduction of CMDM} is to recall the main definitions and also reformulate our main results in terms of weak* convergence of measures on the sets of vertices and edges of $\Gamma\backslash\Tcali$. We start in Section \ref{Drinfeldmodules}, where we provide an overview of the theory of Drinfeld modules of rank 2 over $\C_\infty$ and explain the uniformization of their classifying space $Y(\C_\infty)$ by $\Gamma\backslash\Omega$. In section \ref{S CM}, we review some facts about CM Drinfeld modules. In particular, we recall the parametrization of $\mathrm{CM}(D)$ by the class group $\mathrm{Cl}(\mathcal{O}_D)$ of $\mathcal{O}_D$.

In Section \ref{S tree and building map}, we introduce the Bruhat-Tits tree $\Tcali$ and the so-called building map $\lambda\colon \Omega\to \Tcali(\R)$, whose fibers agree with the fibers of the reduction $\pi\colon \Omega\to\overline{\Omega}$. Here, $\Tcali(\R)$ denotes the realization of $\Tcali$. The main reference for this section is \cite{gekelerreversat}. In section \ref{S reduction of CM points}, we exploit the connection between $\lambda$ and $\pi$ to study the reduction of CM points. In particular, we prove our previous claim that CM points associated to inert discriminants reduce to a unique irreducible component $C_n$, whereas those associated to odd discriminants are mapped to intersection points. 

Let $\mu(v_0):=(q-1)/2q$ and for $n\geq1$, let $\mu(v_n):=(q^2-1)/2q^{n+1}$ and $\nu(e_n):=q^{-n}(q-1)$. Then $\mu$ (resp. $\nu$) defines a probability measure on the set of vertices (resp. edges) of $\Gamma \backslash \Tcali$. Using the results of Section \ref{S reduction of CM points}, in Section \ref{measures},
we reformulate our main results as Theorem \ref{thmA'}, Theorem \ref{thmB'} and Theorem \ref{thmC'}. These are equivalent to Theorem \ref{thmA}, Theorem \ref{thmB} and Theorem \ref{thmC} respectively, but formulated in terms of $\mathrm{Cl}(\mathcal{O}_D)$, the building map $\lambda$, and estimates for certain averages (usually called \textit{Weyl sums}) attached to elements in the $L^2$ space associated to $\mu$ and $\nu$. For example, Theorem \ref{thmA} is equivalent to the statement that for every $f\in L^2(\mu)$, $\varepsilon>0$ and a sequence of inert discriminants $D$ as in Theorem \ref{thmA}, \begin{equation}\label{testexample}\left|\frac{1}{\#\mathrm{Cl}(\mathcal{O}_D)}\sum_{[\afrak]\in \mathrm{Cl}(\mathcal{O}_D)}f\left(\Gamma\lambda(z_\afrak)\right)- \int_{} fd\mu\right|\ll_f|D|^{-1/4+\varepsilon}\quad\mbox{ as }|D|=q^{\deg D}\to\infty.\end{equation} where $z_\afrak$ is any element in $\Omega$ with $A+Az_{\mathfrak{a}}=\mathfrak{a}$ in $\mathrm{Cl}(\mathcal{O}_D)$. If we use $\delta_\phi$ to denote the Dirac measure supported on $\phi\in Y(\C_\infty)$, \eqref{testexample} says that the pushforward of $\frac{1}{\#\mathrm{CM}(D)}\sum_{\phi\in\mathrm{CM}(D)}\delta_\phi$ under $\mathrm{red}$ converges to $\mu$ in the weak* topology.

In Section \ref{sectionspectraldecomposition}, we reduce the proof of the reformulated theorems to the study of the Weyl sums associated to a certain class of functions, the eigenfunctions of the \textit{adjacency operator}. This is done by providing an explicit spectral decomposition for the adjacency operator acting on vertices and edges. For example, every function in $L^2(\mu)$ can be written as \[f(v_n)=\int fd\mu+\langle f,u_\mathrm{alt}\rangle u_\mathrm{alt}(v_n)+\frac{4q}{(q^2-1)}\int_{0}^{\pi}\langle f,E_{\v}(\cdot,\frac{1}{2}+i\frac{\theta}{\log q})\rangle E_{\v}(v_n,\frac{1}{2}+i\frac{\theta}{\log q})d\theta,\] where $u_\mathrm{cst}(v_n)=(-1)^n$ and for $s\in \C$, $E_{\v}(v_n,s)$ is a certain Eisenstein series. This allows us to focus on \eqref{testexample} only in the cases when $f$ equals $u_\mathrm{cst}$ or $E_{\v}(v_n,1/2+i\theta/\log(q))$. On the other hand, the spectral decomposition of the adjacency operator on $L^2(\nu)$ will enable us to focus only on the study of certain Eisenstein series $E_{\e}(e_n,s)$ (see Theorem \ref{Thm spectral resolution on edges}). The spectral decomposition over $L^2(\mu)$ can be found in \cite{efrat} and we explain it in Section \ref{S spectral decomposition vertices}. The same reference serves as a guide to the case of $L^2(\nu)$ which is treated in Section \ref{S spectral decomposition edges}. Section \ref{S alternatingfn} is an interlude where we treat the Weyl sum of the alternating function $u_\mathrm{alt}$.

At this point, the proof of the main theorems boils down to ha ving the right estimates for the Weyl sums of the Eisenstein series $E_{\v}(v_n,s)$ and $E_{\e}(e_n,s)$. This is done in a unified way by using Eisenstein series on $\Gamma\backslash\Omega$. In Section \ref{sectionfundamentaldiscriminant}, we prove Theorem \ref{thmA'} and Theorem \ref{thmC'} in the case of having a sequence of square-free discriminants $D$. In doing so, we use Lindel\"of-type bounds from \cite{bucur2018traces} and \cite{Diaconu}. Finally, mimicking \cite{clozelullmo}, in Section \ref{sectiongeneraldiscriminant} we introduce Hecke operators to prove Theorem \ref{thmB'} and extend Theorems \ref{thmA'} and \ref{thmC'} to general discriminants.



\section{Reduction of CM Drinfeld modules}\label{S reduction of CMDM}
\subsection{Drinfeld modules of rank $2$ over $\C_\infty$}\label{Drinfeldmodules}

We refer to \cite{papikianlibro} for general background on the theory of Drinfeld modules. Recall that a Drinfeld $A$-module of rank $2$ over a field extension $L$ of $k$ is defined by a twisted polynomial \[\phi=T+g\tau+\Delta \tau^2\in L\{\tau\},\] where $\Delta\neq0$. The Drinfeld module $\phi$ is said to be CM if its endomorphism ring \[\mathrm{End}(\phi)=\{f\in L\{\tau\}\mid f\phi=\phi f\}\] is strictly larger than $A$, in which case it is an $A$-order in an imaginary quadratic extension of $k$, that is, an extension that does not split at $\infty$.

Similarly to the case of elliptic curves over $\C$, Drinfeld $A$-modules over $\C_\infty$ admit an analytic description in terms of lattices. More precisely, given a rank $2$ $A$-lattice $\Lambda\subseteq\C_\infty$, there exists an analytic isomorphism $e_\Lambda\colon \C_\infty/\Lambda\to\C_\infty$ and a twisted polynomial $\phi_\Lambda=T+g(\Lambda)\tau+\Delta(\Lambda)\tau^2\in\C_\infty\{\tau\}$ such that the following diagram commutes. 

\begin{center}
\begin{tikzcd}
\C_\infty/\Lambda \arrow[r, "\times T"] \arrow[d, "e_\Lambda"] & \C_\infty/\Lambda \arrow[d, "e_\Lambda"]\\ \C_\infty \arrow[r,"\phi_\Lambda"]
& \C_\infty.
\end{tikzcd}
\end{center}

By \cite{drinfeld74}, the assignment $\Lambda\mapsto\phi_\Lambda$ defines a bijection between homothety classes of rank $2$ $A$-lattices in $\C_\infty$ and $Y(\C_\infty)$. Each homothety class admits a representative of the form $\Lambda_z:=A+zA$ with $z\in \Omega$ and $A+zA=A+z'A$ if and only if $\Gamma z=\Gamma z'$. Therefore, we obtain the identification \[\Gamma\backslash\Omega\cong Y(\C_\infty).\]

For $z\in \Omega$, its \textit{imaginary part} is defined by \[|z|_i=\min_{w\in k_\infty}|z-w|.\]

By \cite[Proposition 6.5]{gekeler99}, the set \[\mathcal{F}=\{z\in \Omega\mid |z|=|z|_i\geq 1\}\] is a fundamental domain for the action of $\Gamma$ on $\Omega$.

\subsection{CM Drinfeld modules over $\C_\infty$}\label{S CM} The endomorphism ring of $\phi_\Lambda$ can be identified with the ring of multipliers $\{\lambda\in \C_\infty\mid \lambda \Lambda\subseteq\Lambda\}$ of $\Lambda$. If $\Lambda$ is homothetic to $\Lambda_z$, then $\phi_\Lambda$ has CM if and only if $k(z)$ is an imaginary quadratic extension of $k$. For $K/k$ an imaginary quadratic extension with $K\subseteq\C_\infty$, we denote by $\mathrm{CM}(K)$ the set of CM Drinfeld modules of rank $2$ over $\C_\infty$ with CM by an order inside $K$. For an $A$-order $\mathcal{O}\subseteq K$, $\mathrm{CM}(\mathcal{O})$ denotes the collection of CM Drinfeld modules of rank $2$ over $\C_\infty$ with CM by $\mathcal{O}$. Recall that a fractional $\mathcal{O}$-ideal $\Lambda\subseteq K$ is proper if its ring of multipliers is exactly $\mathcal{O}$. The class group $\mathrm{Cl}(\mathcal{O})$ is the finite abelian group obtained as the quotient of the group of proper $\Ocali$-ideals by the subgroup of principal ideals. Its cardinality will be denoted by $h(\Ocali)$. It follows that


\begin{proposition}\label{CMDbijection}
    Let $K/k$ be an imaginary quadratic extension with $K\subseteq\C_\infty$ and let $\mathcal{O}\subseteq K$ be an $A$-order. The assignment $\Lambda\mapsto \phi_\Lambda$ induces the following bijections:
    \begin{enumerate}
        \item $A$-lattices $\Lambda\subseteq K$, up to multiplication by $K^\times$ and $\mathrm{CM}(K)$.
        \item The class group $\mathrm{Cl}(\mathcal{O})$ and $\mathrm{CM}(\mathcal{O})$.
    \end{enumerate}
\end{proposition}

Call a non-square polynomial $D\in A$ an imaginary discriminant if $K_D/k$ is an imaginary quadratic extension. Recall that $\mathcal{O}_D:=A[\sqrt{D}]$, and let $\mathrm{CM}(D):=\mathrm{CM}(\mathcal{O}_D)$, $\mathrm{Cl}(D):=\mathrm{Cl}(\mathcal{O}_D)$ and $h(D):=h({\Ocali}_D)$.

\begin{lemma}\label{discriminant}
Let $D$ be an imaginary discriminant and $z\in\Omega$ quadratic over $k$. Then $\phi_{\Lambda_z}\in \mathrm{CM}(D)$ if and only if there exist coprime $a,b,c \in A$ such that $az^2+bz+c=0$ and $D=b^2-4ac$.
\end{lemma}
\begin{proof} Suppose that $\phi_{\Lambda_z}\in \mathrm{CM}(D)$ and let $az^2+bz+c=0$ be any quadratic equation for $z$ with $a,b,c\in A$ coprime. For $\lambda\in\C_\infty$,  $\lambda\Lambda_z\subseteq\Lambda_z$ if and only if there exist $u,v,n,m$ in $A$ such that $\lambda=n+mz$ and $\lambda z=u+vz$. Writing $\lambda z=nz+m\left(\frac{-bz-c}{a}\right)=\frac{-cm}{a}+\left(n-\frac{bm}{a}\right)z$, we conclude that $\lambda\Lambda_z\subseteq \Lambda_z$ if and only if $\lambda\in \Lambda_{az}$ because $a$ is coprime to both $b$ and $c$. Since we are assuming $2\in\F_q^\times$, $\Ocali_D=\Lambda_{az}=A[\frac{-b\pm\sqrt{b^2-4ac}}{2}]=A[-b+\sqrt{b^2-4ac}]=A[\sqrt{b^2-4ac}]$. Thus, $\phi_{\Lambda_z}\in\mathrm{CM}(D)$ if and only if $b^2-4ac$ is of the form $D\alpha^2$ for some $\alpha\in\F_q^\times$. Multiplying $(a,b,c)$ by $\alpha^{-1}$ allows us to conclude.   
\end{proof}

For $D$ an imaginary discriminant, we denote by $\Omega_D$ the subset of $z\in \Omega\cap K_D$ satisfying a quadratic equation $az^2+bz+c=0$ with $a,b,c\in A$ coprime and $D=b^2-4a$. If $\mathfrak{a}$ is a proper $\mathcal{O}_D$-ideal, up to the action of $\Gamma$, there exists a unique $z_\mathfrak{a}\in \Omega_D$ such that $[\mathfrak{a}]=[\Lambda_{z_\mathfrak{a}}]$ in $\mathrm{Cl}(\Ocali_D)$. The map sending $[\afrak]\in \mathrm{CM}(D)$ to the isomorphism class of $\phi_{z_\mathfrak{a}}:=\phi_{\Lambda_{z_\mathfrak{a}}}$ in $\mathrm{CM}(D)$ is well-defined and it is a bijection by Proposition \ref{CMDbijection}. We can always assume that $z_\mathfrak{a}\in\mathcal{F}$.

\subsection{The Bruhat-Tits tree and the building map}\label{S tree and building map}

Recall that the Bruhat-Tits tree $\Tcali$ of $\PGL_2(k_\infty)$ is the $(q+1)$-regular tree constructed as follows. Two rank-$2$ $\mathcal{O}_\infty$-modules $L$ and $L'$ in $k_\infty^2$ are \textit{equivalent} if there exists $x\in k_\infty^\times$ such that $L'=xL$. Write $[L]$ for the equivalence class of $L$. Two classes $[L]$ and $[L']$ are \textit{adjacent} if there exists $L''\in[L']$ such that $L''\subseteq L$ and $L/L'' \cong\mathbb{F}_q$. Then $\Tcali$ is the combinatorial graph whose vertices $\mathcal{V}(\Tcali)$ are the classes $[L]$ and two vertices are connected if the respective classes are adjacent. We denote by $\mathcal{E}(\Tcali)$ the set of edges in $\Tcali$, i.e. the set of adjacent vertices.

We consider $k_\infty^2$ as row vectors with a right action of $\GL_2(k_\infty)$ given by multiplication on the right. It induces a natural left action of $\PGL_2(k_\infty)$ on $\mathcal{V}(\mathcal{T})$ given by $g[L]=[Lg^{-1}]$. This action respects the connections in $\mathcal{V}(\Tcali)$, inducing also a left action on $\Tcali$. The action of $\PGL_2(k_\infty)$ is  transitive on both $\mathcal{V}(\Tcali)$ and $\mathcal{E}(\Tcali)$. Consider the lattice $L_0:=\mathcal{O}_\infty^2$ and its sublattice $L_1:=\begin{pmatrix}
    T&0\\0&1
\end{pmatrix}L_0=L_0\begin{pmatrix}
    T^{-1}&0\\0&1
\end{pmatrix}=T^{-1}\mathcal{O}_\infty\times\mathcal{O}_\infty$. Denote $\boldsymbol{v}=[L_0]$ and $\boldsymbol{e}=\{[L_0],[L_1]\}$. The stabilizer of $\boldsymbol{v}$ in $\PGL_2(k_\infty)$ is the maximal open compact subgroup $\K:=\PGL_2(\mathcal{O}_\infty)$ and the stabilizer of $\boldsymbol{e}$ is the subgroup $\mathbb{I}$ generated by the Iwahori subgroup $\Gamma_0(T^{-1})=\left\{\begin{pmatrix}
    a&b\\c&d
\end{pmatrix}\in\K\mid c\in T^{-1}\mathcal{O}_\infty\right\}$ (which fixes $[L_0]$ and $[L_1]$) and the element $w_\pi=\begin{pmatrix}
    0&1\\T^{-1}&0
\end{pmatrix}$ (which interchanges $[L_0]$ and $[L_1]$). From this we obtain two uniformizations: \[\PGL_2(k_\infty)/\mathbb{K}\cong \mathcal{V}(\Tcali),\quad\mbox{induced by }g\mapsto g\boldsymbol{v},\] and \[\PGL_2(k_\infty)/\mathbb{I}\cong \mathcal{E}(\Tcali),\quad\mbox{induced by }g\mapsto g\boldsymbol{e}.\]

The geometric realization $\Tcali(\R)$ of $\Tcali$ is the topological space obtained by attaching a unit interval to each non-oriented edge of $\Tcali$, with endpoints identified whenever the corresponding edges share a vertex. For an edge $e\in\mathcal{E}(\Tcali)$, we denote by $e^\circ$ the open interval of $\Tcali(\R)$ in between the two vertices defining $e$. The space $\Tcali(\R)$ can be described in terms of norms in $k_\infty^2$. Recall that a norm on $k_\infty^2$ is a map $||\cdot||\colon k_\infty^2\to\R_{\geq0}$ satisfying $||(x,y)||=0$ if and only if $x=y=0$, $||\lambda(x,y)||=|\lambda|||(x,y)||$ for all $ \lambda\in k_\infty$, and $||(x,y)+(x',y')||\leq\sup\{||(x,y)||,||(x',y')||\}$. We say that two norms on $k_\infty^2$ are similar if they differ by a real multiplicative constant. For $g\in \GL_2(k_\infty)$, the rule that sends a norm $||\cdot||$ to $(x,y)\mapsto||(x,y)g||$ induces left action of $\PGL_2(k_\infty)$ on the similarity classes of norms. 

Associate to each vertex $[L]$ the similarity class of the norm \[||(x,y)||_L:=\inf\{|z|: z\in k_\infty\mbox{ and }(x,y)\in zL\}=q^{-n},\] where $n=\sup\{m\in\Z\mid (x,y)\in T^{-m}L\smallsetminus T^{-(m+1)}L\}$. Now consider an oriented edge $e^+=([L],[L'])$ with $T^{-1}L'\subseteq L\subseteq L'$. Let $0<t<1$ and $e_t^+=(1-t)[L]+t[L']\in e^\circ$. We assign to $e_t^+$ the class of the norm \[||(x,y)||_{e_t^+}=\sup\{|(x,y)|_L,q^{t}|(x,y)|_{L'}\}.\]

\begin{theorem}\label{buildintthm}\hfill
\begin{enumerate}
    \item The previous construction induces a $\PGL_2(k_\infty)$-equivariant bijection between $\Tcali(\R)$ and the space of similarity classes of norms on $k_\infty^2$.
    \item Under the bijection in \textup{(1)}, the map $\lambda\colon \Omega\to \Tcali(\R)$ induced by $\lambda(z)(x,y)=|xz+y|$ is $\PGL_2(k_\infty)$-equivariant and for each irreducible component $C$ of $\overline{\Omega}$, there exists a unique vertex $[L_C]$ of $\Tcali$ such that \begin{equation}\label{building}
\pi^{-1}(C\smallsetminus C(\F_q))=\lambda^{-1}([L_C]).
\end{equation}
\item For two irreducible components $C$ and $C'$ of $\overline{\Omega}$ that intersect, \begin{equation}\label{buildingedge}\pi^{-1}(C\cap C')=\lambda^{-1}(e^\circ),\end{equation} where $e=\{[L_C],[L_{C'}]\}$.
\item For each edge $e=\{[L],[L']\}$, there exists a unique $n\in\Z$ and $x\in k_\infty \mod \pi^{n+1}\mathcal{O}_\infty$ such that $\lambda^{-1}(e^\circ)=\{z\in\Omega\mid q^{-(n+1)}<|z-x|<q^{-n}\}$ and possible after switching the role of $[L]$ and $[L']$, $\lambda^{-1}([L])=\{z\in\Omega\mid |z-x|=|z-x|_i=q^{-n}\}$ and  $\lambda^{-1}([L'])=\{z\in\Omega\mid |z-x|=|z-x|_i=q^{-(n+1)}\}$
\end{enumerate}
\end{theorem}

\begin{proof}
    Items (1), (2), (3) and (4) can be found in Sections 1.4 and 1.5 of\cite{gekelerreversat}.
\end{proof}

The map $\lambda$ is called the \textit{building map}. It will be necessary for our purposes to specialize (4) to the case of our distinguished vertex $\boldsymbol{e}$. Let $\boldsymbol{e}^+=([L_0],[L_1])$ be an orientation for $\boldsymbol{e}$.

\begin{lemma}\label{distinguishedred}
Let $0<t<1$ and $z\in\Omega$. If $|z|=q^t$ \begin{equation}\label{normramified}||(x,y))||_{\boldsymbol{e}_t^+}=\sup\{q^t|x|,|y|\}=|zx+y|,\end{equation} and then $\lambda(z)=\boldsymbol{e}_t^+$. If $|z|=|z|_i=1$ \begin{equation}\label{norminert}
    ||(x,y)||_{L_0}=\sup\{|x|,|y|\}=|zx+y|,
\end{equation} and then $\lambda(z)=\v$. In particular $\lambda^{-1}(\boldsymbol{e}^\circ)=\{z\in \Omega\mid 1<|z|<q\}$ and $\lambda^{-1}(\v)=\{z\in \Omega\mid |z|=|z|_i=1\}$.
\end{lemma}

\begin{proof}
We have that $||(x,y)||_{L_0}=\sup\{|x|,|y|\}$ and \[||(x,y)||_{TL_1}=||(x,y)||_{\begin{pmatrix}
    1&0\\0&T^{-1}
\end{pmatrix}L_0}=||(x,T^{-1}y)||_{L_0}=\sup\{|x|,q^{-1}|y|\}.\] Let $0<t<1$. Since $L_1\subseteq L_0\subseteq TL_1$,  we have \[||(x,y)||_{\boldsymbol{e}_t^+}=\sup\{\sup\{|x|,|y|\},q^t\sup\{|x|,q^{-1}|y|\}\}.\]

For $|x|\geq|y|$, \[||(x,y)||_{\boldsymbol{e}_t^+}=\sup\{|x|,q^t|x|\}=q^t|x|=\sup\{q^t|x|,|y|\}.\] In the case $|x|<|y|$, $q|x|\leq|y|$ and then \[||(x,y)||_{\boldsymbol{e}_t^+}=\sup\{|y|,q^{t-1}|y|\}=|y|=\sup\{q^t|x|,|y|\}.\]

Now if $|z|=q^t$ we have that $q^t|x|\neq|y|$. Then $|zx+y|=\sup\{q^t|x|,|y|\}$, proving \eqref{normramified}. If follows $\lambda^{-1}([L_0])$ equals $\{z\in \Omega\mid |z|=|z|_i=1\}$ or $\{z\in \Omega\mid |z|=|z|_i=q\}$. To prove \eqref{norminert}, take $|z|=|z|_i=1$. Then if $|x|\neq|y|$, $|zx+y|=\sup\{|x|,|y|\}$. If $|x|=|y|\neq0$, $|zx+y|\leq\sup\{|x|,|y|\}$ and $|zx+y|=|x||z-y/x|\geq |x||z|_i=|x|$. In conclusion, $|zx+y|=||(x,y)||_{L_0}$.
\end{proof}

\subsection{Reduction of CM points}\label{S reduction of CM points}
The assignment $C\mapsto [L_C]$ induces a canonical identification between the intersection graph $\Tcal$ and $\Tcali$ which is compatible with the $\PGL_2(k_\infty)$-action on both sides. Denote by $\overline{\lambda}$ the induced map $\Gamma\backslash\Omega\to \Gamma\backslash\Tcali(\R)$. The set of vertices of $\Gamma\backslash\Tcali$ are in bijection with $\Gamma\backslash\mathcal{V}(\Tcali)\cong \Gamma\backslash \PGL_2(k_\infty)/\mathbb{K}$. For $n\geq0$, consider $g_n=\begin{pmatrix}
    T^n&0\\0&1
\end{pmatrix}\in\PGL_2(k_\infty)$. From \cite[Example 2.4.1]{serretrees}, one has the decomposition
\[\Gamma\backslash\PGL_2(k_\infty)=\bigsqcup_{n\geq0}\Gamma\backslash\Gamma g_n\K.\] Let $n\geq0$. the vertex $v_n$ of \eqref{line} corresponds to the $\Gamma$-orbit of $g_n\boldsymbol{v}$ in $\Gamma\backslash\Tcali$. Likewise, if we denote by $e_{n+1}$ the edge connecting $v_{n}$ with $v_{n+1}$, then $e_{n+1}$ corresponds to the $\Gamma$-orbit of $g_{n}\boldsymbol{e}$ in $\Gamma\backslash\Tcali$. In terms of $\Gamma\backslash\Tcal$ we obtain that a CM Drinfeld module $\phi\cong\phi_z$ reduces to a unique irreducible component $C_n$ of $\Gamma\backslash\overline{\Omega}$ if and only if and $\overline{\lambda}(z)=v_n$. Similarly, $\phi$ reduces to the intersection $c_{n+1}=C_n\cap C_{n+1}$ if and only if $\overline{\lambda}(z)\in e_{n+1}^\circ$.

In what follows $k_{\infty^2}=\F_{q^2}((T^{-1}))\subseteq\C$ denotes the unique unramified quadratic extension of $k_\infty$ and $k_{ram}\subseteq\C_\infty$ denotes a ramified quadratic extension of $k_\infty$. Fix $u\in \F_{q}\smallsetminus (\F_q)^2$ and $\i\in\F_{q^2}$ with $\i^2=u$. Then $k_{\infty^2}=k_\infty(\i)$ and $k_{ram}=k_\infty(\r)$ with $\r=\sqrt{T}$ or $\r=\sqrt{uT}$.

\begin{lemma}\label{imaginary1}
    Let $g=\begin{pmatrix}a&b\\c&d\end{pmatrix}\in \PGL_2(k_\infty)$. Then 
    \begin{enumerate}
        \item If $z\in k_{\infty^2}\smallsetminus k_\infty$ and  $z=x+y\i$ with $x,y\in k_\infty$ and $y\neq0$. Then $|z|_i=|y|$ and $|g z|_i=\frac{|\det(g)||z|_i}{|cz+d|^2}$.
        \item If $z\in k_{ram}\smallsetminus k_\infty$ and  $z=x+y\r$ with $x,y\in k_\infty$ and $y\neq0$. Then $|z|_i=|y\r|=q^{1/2}|y|$ and $|g z|_i=\frac{|\det(g)||z|_i}{|cz+d|^2}$.
    \end{enumerate} 
\end{lemma}
\begin{proof}
We only prove (1), and (2) is proven in an analogous way. We have \[|z|_i=\min_{w\in k_\infty}|(x-w)+y\i|=\min_{x'\in k_\infty}\max\{|x'|,|y|\}=|y|,\] where the second equality follows from Lemma \ref{distinguishedred}. For the second claim, a direct computation shows that there exists $x'\in k_\infty$ such that \begin{equation}\label{imag}g z:=\frac{az+b}{cz+d}=\frac{x'+\det(\gamma)y\i}{N_{k_{\infty^2}/k_\infty}(cz+d)}\in k_{\infty^2}\smallsetminus k{_\infty}.\end{equation} and we conclude by using the first part. 
\end{proof}

\begin{lemma}\label{cmred}
Let $D$ be an imaginary discriminant and $z\in\Omega_D\cap\mathcal{F}$. If $D$ is an inert discriminant, $\mathrm{red}(\phi_{\Lambda_z})$ belongs to a unique irreducible component $C_n$ with $n\geq0$ and $|z|=q^n$. If $D$ is odd, $\mathrm{red}(\phi_{\Lambda_z})\in c_n$ with $n\geq0$ and and $|z|=q^{n+1/2}$. 
\end{lemma}

\begin{proof}
First assume that $D$ is inert and then $K_D\otimes k_\infty=k_{\infty^2}$. If $|z|=q^n$, $|z/T^n|=|z/T^n|_i=1$ and by Lemma \ref{distinguishedred}, $\lambda(z)=g_n\lambda(z/T^n)=g_n\boldsymbol{v}$. Then $\mathrm{red}(\phi_z)=\overline{\lambda}(z)=v_n$. If $D$ is odd, $K_D\otimes k_\infty=k_{ram}$ and then $|z|=|z|_i=q^{1/2}|y|$ where $z=x+y\r$. By Lemma \ref{distinguishedred}, $\lambda(z)=g_n\lambda(z/T^n)=g_n\boldsymbol{e}$. From this $\mathrm{red}(\phi_z)=g_n\overline{\lambda}(z)=(e_{n+1})_{1/2}$.
\end{proof}

\subsection{Measures on $\Gamma\backslash\Tcali$}\label{measures}

The set of vertices $\Gamma\backslash\mathcal{V}(\Tcali)$ carries a natural measure arising from the uniformization by the double quotient $\Gamma\backslash \PGL_2(k_\infty)/\mathbb{K}$. A Haar measure $\rho$ on $\PGL_2(k_\infty)$, together with the counting measure $\lambda$ on $\Gamma$, induces a unique measure $\mu_{\Gamma\backslash \PGL_2(k_\infty)}$ on $\Gamma\backslash\PGL_2(k_\infty)$ such that \begin{equation}\label{unfolding}\int_{\Gamma\backslash \PGL_2(k_\infty)}\int_\Gamma f(\gamma g)d\lambda(\gamma)d\mu_{\Gamma\backslash \PGL_2(k_\infty)}(g)=\int_{\PGL_2(k_\infty)} f(g)d\rho(g),\end{equation} for all $f\in C_c(\PGL_2(k_\infty))$. See \cite[\S 2 Proposition 4]{Bourbaki04}. 

\begin{lemma}
    If $\rho$ is normalized such that $\rho(\K)=1$, $\mu_{\Gamma\backslash \PGL_2(k_\infty)}(\Gamma\backslash\Gamma g_n\K)=\left|\Gamma\cap g_n\K g_n^{-1}\right|^{-1}$.
\end{lemma}
\begin{proof}
We denote by $\mathbbm{1}_X$ the characteristic function of a set $X$. For $n\geq0$, consider the function $f_n=|\Gamma\cap g_n \K g_n^{-1}|^{-1}\mathbbm{1} _{g_n\K}\in C_c(G)$. We compute $\int_\Gamma f_n(\gamma g)d\lambda(\gamma)$ for $g\in G$. Suppose that $g\in \Gamma g_m\K$ for some $m\geq0$. Writing $g=\gamma' g_m k$ for some $\gamma'\in\Gamma$ and $k\in \K$, we obtain \[\int_\Gamma f_n(\gamma g)d\lambda(\gamma)=\frac{1}{|\Gamma\cap g_n \K g_n^{-1}|}\int_\Gamma \mathbbm{1}_{g_n\K}(\gamma g_mk )d\lambda(\gamma)=\frac{1}{|\Gamma\cap g_n \K g_n^{-1}|}\delta_{n=m}\int_{\Gamma\cap g_n\K g_n^{-1}}d\lambda(\gamma)=\mathbbm{1}_{\Gamma\backslash \Gamma g_n \K}.\] 

Therefore by (\ref{unfolding}) \begin{equation}\label{vol}\mu_{\Gamma\backslash G}(\Gamma\backslash \Gamma g_n\K)=\int_{\PGL_2(k_\infty)} f_n(g)d\rho(g)=\frac{\rho(g\K)}{|\Gamma\cap g_n \K g_n^{-1}|}=|\Gamma\cap g_n \K g_n^{-1}|^{-1}.\end{equation}    
\end{proof}

Let $\mu$ be the unique probability measure on $\Gamma\backslash\mathcal{V}(\Tcali)$ associated with the pushforward of $\mu_{\Gamma\backslash \PGL_2(k_\infty)}$ to $\Gamma\backslash \PGL_2(k_\infty)/\mathbb{K}$. Since $\Gamma\cap \K=\PGL_2(\F_q)$, and for $n\geq0$ the group $\Gamma\cap g_n\K g_n^{-1}$ equals the projective image in $\PGL_2(k_\infty)$ of $\begin{pmatrix}\F_q^\times&\F_q[t]_{\leq n}\\0&\F_q^\times\end{pmatrix}$, it follows that \begin{center}
    $$\mu(v_n)=\begin{cases}
        \dfrac{q-1}{2q}&\text{ if } n=0\\
        \dfrac{q^2-1}{2q^{n+1}}& \text{ if }n\ge1.
    \end{cases}$$
\end{center} 

Theorem \ref{thmA} and Theorem \ref{thmB} follows as a direct consequence of the following results, respectively.

\begin{customtheorem}{A'}[]\label{thmA'}
For every $\varepsilon>0$, $f\in L^2(\Gamma\backslash\mathcal{V}(\Tcali))$, and a sequence of inert discriminants $D$ as in Theorem \ref{thmA} \begin{equation}\label{testA}\left|\frac{1}{h(D)}\sum_{[\afrak]\in \mathrm{Cl}(D)}f\left(\overline{\lambda}(z_\afrak)\right)- \int_{\Gamma\backslash\mathcal{V}(\Tcali)} fd\mu\right|\ll_f|D|^{-1/4+\varepsilon}\quad\mbox{ as }|D|=q^{\deg D}\to\infty.\end{equation}   
\end{customtheorem} 

\begin{customtheorem}{B'}[]\label{thmB'}
Let $f\in L^2(\Gamma\backslash\mathcal{V}(\Tcali))$ be a function supported on vertices $v_n$ with $n$ even (resp. odd). For every sequence of inert discriminants $D$ as in Theorem \ref{thmB} with $\deg f$ even (resp. odd) and $\varepsilon>0$, \begin{equation}\label{testB}\left|\frac{1}{h(D)}\sum_{[\afrak]\in \mathrm{Cl}(D)}f\left(\overline{\lambda}(z_\afrak)\right)- 2\int_{\Gamma\backslash\mathcal{V}(\Tcali)} fd\mu\right|\ll_f|D|^{-1/4+\varepsilon}\quad\mbox{ as }|D|=q^{2\deg f}\to\infty.\end{equation}   
\end{customtheorem}

Indeed, by Lemma \ref{cmred} and the bijection $[\afrak]\mapsto \phi_{z_\mathfrak{a}}$ between $\mathrm{Cl}(D)$ and $\mathrm{CM}(D)$, we recover the limits in Theorems \ref{thmA} and \ref{thmB} as a special case when $f$ is the characteristic function supported on $v_n$ with $n\geq0$.

Similarly, if we uniformize the edges $\Gamma\backslash\mathcal{E}(\Tcali)$ by the double quotient $\Gamma\backslash\PGL_2(k_\infty)/\mathbb{I}$, we can normalize $\rho$ such that $\rho(\I)=1$ and this time $\mu_{\Gamma\backslash \PGL_2(k_\infty)}(\Gamma\backslash\Gamma g_n\mathbb{I})=\left|\Gamma\cap g_n\mathbb{I} g_n^{-1}\right|^{-1}$. Since \[\Gamma\cap g_n\mathbb{I}g_{n}^{-1}=\Gamma\cap g_n\Gamma_0(T^{-1})g_n^{-1}=\Gamma\cap g_n\K g_n^{-1}\cap g_{n+1}\K g_{n+1}^{-1},\] the pushforward of $\mu_{\Gamma\backslash \PGL_2(k_\infty)}$ to $\Gamma\backslash\PGL_2(k_\infty)/\mathbb{I}$ induces a unique probability measure $\nu$ on $\Gamma\backslash\mathcal{E}(\Tcali)$ given by

\begin{center}
    $$\nu(e_{n+1})=\dfrac{q-1}{q^{n+1}},\quad n\ge0$$
\end{center}

For $f,g\colon \Gamma\backslash\mathcal{E}(\Tcali)\to\C$, we put \[\langle f,g\rangle:=\int_{\Gamma\backslash\mathcal{E}(\Tcali)}f(e_n)\overline{g(e_n)}d\mu(e_n).\] 

In what follows, we identify the middle point of an edge with the edge itself, so that from now on, if $z\in k_{ram}\smallsetminus k_\infty$ we consider $\lambda(z)$ and $\overline{\lambda}(z)$ as elements in $\mathcal{E}(\Tcali)$ and $\mathcal{E}(\Gamma\backslash\Tcali)$ respectively. Note that for any $0<t<1$ with $t\neq1/2$, one would first need to consider an orientation of the edge if one wishes to identify it with a point at distance $t$ from one of its vertices. Having done this, the second claim in Theorem \ref{thmC} is a direct consequence of 

\begin{customtheorem}{C'}[]\label{thmC'}
For every $\varepsilon>0$, $f\in L^2(\Gamma\backslash\mathcal{E}(\Tcali))$ and a sequence of odd discriminants $D$ as in Theorem \ref{thmC} \begin{equation}\label{testC}\left|\frac{1}{h(D)}\sum_{[\afrak]\in \mathrm{Cl}(D)}f\left(\overline{\lambda}(z_\afrak)\right)- \int_{\Gamma\backslash\mathcal{E}(\Tcali)} fd\mu\right|\ll_f|D|^{-1/4+\varepsilon}\quad \mbox{ as }|D|=q^{\deg D}\to\infty.\end{equation}
\end{customtheorem} 

Again this follows from considering $f$ as the characteristic function of $e_{n+1}$, the bijection between $\mathrm{Cl}(D)$ and $\mathrm{CM}(D)$, and Lemma \ref{cmred}.

\section{Spectral decompositions}\label{sectionspectraldecomposition}

In this section, we introduce the adjacency operator on vertices and edges with the aim of studying the spectral decomposition of the induced operator on the $L^2$-spaces previously defined. In this way, we will reduce the proof of Theorems \ref{thmA'}, \ref{thmB'} and \ref{thmC'} to a particular family of funcions $f$. For the case of $L^2(\Gamma\backslash\mathcal{V}(\Tcali))$, the main references are \cite{efrat} and \cite[Section 2]{nagoshi}. The exposition for $L^2(\Gamma\backslash\mathcal{E}(\Tcali))$ parallel that of \cite{efrat}.

\subsection{Spectral decomposition on $L^2(\Gamma\backslash\mathcal{V})$}\label{S spectral decomposition vertices}

Let $\T$ be the adjacency operator on $\mathcal{V}(\Tcali)$ defined by \[\T(f)(v)=\sum_{v'\mbox{ is adjacent to }v}f(v').\] It descends to a self-adjoint operator on $\Gamma\backslash\mathcal{V}(\Tcali)$ given by \[\T(f)(v_n)=\begin{cases}(q+1)f(v_1),&\mbox{ if }n=0\\qf(v_{n-1})+f(v_{n+1}),&\mbox{ if }n\geq1.\end{cases}\]

On the quotient graph, the discrete spectrum of $\T$ consists of the eigenvalues $q+1$ and $-(q+1)$. The eigenfunctions corresponding to $q+1$ are the constant functions, and those corresponding to $-(q+1)$ are multiples of the alternating function $u_\mathrm{alt}(v_n):=(-1)^n$. The continuous spectrum is given by the interval $[-2\sqrt{q},2\sqrt{q}]$. To describe the corresponding eigenfunctions we introduce Eisenstein series. 

Let $\Gamma_\infty$ denote the subgroup of upper triangular matrices in $\Gamma$. For $(x,y)\in k_\infty^2$ and $g\in \GL_2(k_\infty)$ we define the right $\mathbb{K}$-invariant function $\psi_{\boldsymbol{v}}(g)=|\det(g)|||(0,1)g||_{L_0}^{-2}$. For $s\in\C$, we define the Eisenstein series \[E_{\v}(g,s)=\sum_{\gamma\in \Gamma_\infty\backslash \Gamma}\psi_{\v}(\gamma g)^s.\] For fixed $s$, the function $E_{\v}(g,s)$ is left $\Gamma$-invariant and right $\mathbb{K}$-invariant. Therefore, it defines a function $E_{\v}(v_n,s):=E_{\v}(g_n,s)$ on $\Gamma\backslash \mathcal{V}(\Tcali)$. For fixed $g$, $E_{\v}(g,s)$ is a rational function of $q^{-s}$. Moreover, it is holomorphic on $\mathrm{Re}(s)\geq 1/2$ except for simple poles at $s=1+k\frac{\pi i}{\log q}$ with $k\in\Z$.

The Eisenstein series satisfy \begin{equation}\label{eiseneigenvalue}\T(E_{\v}(\cdot,s))(g)=(q^s+q^{1-s})E_{\v}(g,s).\end{equation} Note that $q^s+q^{1-s}$ is real if and only if $s=1/2+it$ with $ t\in\R$ or $s=1+\frac{k\pi i}{\log(q)}$ with $k\in \Z$. In the first case, if $t=\frac{\theta}{\log q}$ we have $q^s+q^{1-s}=2\sqrt{q}\cos(\theta)$. 

\begin{lemma}\label{Eisenatzero} The value of
$E_{\v}(v,s)$ at the vertex $v_0$ is $\frac{(q+1)(1-q^{1-2s})}{1-q^{2(1-s)}}=(q+1)\frac{\zeta_A(2s-1)}{\zeta_A(2s)}$.
\end{lemma}
\begin{proof}
We need to compute \[E_{\v}(v_0,s)=\frac{1}{q-1}\sum_{(c,d)=A}\max\{|c|,|d|\}^{-2s}.\]

Denote by $A_+$ the set of monic polynomials in $A$. It suffices to compute \[E^\star_{\v}(v_0,s):=\sum_{(0,0)\neq(c,d)}\max\{|c|,|d|\}^{-2s},\]
as we can see from the identity
 \begin{equation}\label{atzero}
    E^\star_{\v}(v_0,s)=\sum_{f\in A_+}\sum_{(0)\neq(c,d)=fA}\max\{|c|,|d|\}^{-2s}=\zeta_A(2s)(q-1)E_{\v}(v_0,s)\end{equation}
We decompose the sum as \begin{equation}\label{sum1}
    E^\star_{\v}(v_0,s)=\sum_{0\neq d}|d|^{-2s}+\sum_{c\neq 0,d\in A}\max\{|c|,|d|\}^{-2s}.\end{equation}

For the first sum in \eqref{sum1} we have \begin{equation}\label{sum11}\sum_{0\neq d}|d|^{-2s}=(q-1)\zeta_A(2s).\end{equation}

We split the second sum in \eqref{sum1} as \begin{equation}\label{sum2}
    \sum_{c\neq 0,d\in A}\max\{|c|,|d|\}^{-2s}=\sum_{n\geq0}\sum_{\deg c=n}\left(\sum_{\deg d\leq n}(q^{n})^{-2s}+\sum_{\deg d\geq n+1}|d|^{-2s}\right).\end{equation}

Computing the two summands in the innermost part of \eqref{sum2} separately, we obtain \[\sum_{\deg d\leq n}(q^{n})^{-2s}=q^{n+1}(q^{n})^{-2s}=q(q^{1-2s})^n,\] and \[\sum_{\deg d\geq n+1}|d|^{-2s}=\sum_{m\geq n+1}\sum_{\deg d=m}(q^m)^{-2s}=\sum_{m\geq n+1}(q-1)(q^{1-2s})^m=(q-1)(q^{1-2s})^{n+1}\frac{1}{1-q^{1-2s}}.\]

Therefore, the second summand in \eqref{sum1} becomes \begin{equation}\label{sum12}\sum_{c\neq 0,d\in A}\max\{|c|,|d|\}^{-2s}=\sum_{n\geq0}(q-1)(q^{2-2s})^n\left(\frac{q(1-q^{-2s})}{1-q^{1-2s}}\right)=\frac{q(q-1)(1-q^{-2s})}{(1-q^{1-2s})(1-q^{2-2s})}.\end{equation}

Adding \eqref{sum11} and \eqref{sum12}, we obtain

\[E^\star(v_0,s)=\frac{q-1}{1-q^{1-2s}}+\frac{q(q-1)(1-q^{-2s})}{(1-q^{1-2s})(1-q^{2-2s})}=\frac{q^2-1}{1-q^{2(1-s)}}.\]

We conclude from \eqref{atzero} that \[E_{\v}(v_0,s)=\frac{(q+1)(1-q^{1-2s})}{1-q^{2(1-s)}}.\]

\end{proof}

Let $u_{\mathrm{cst}}$ be the constant function equal to $1$. 

\begin{theorem}\label{spectraldec}
    The spectral resolution for $f\in L^2(\Gamma\backslash\mathcal{V}(\Tcali))$ reads \[f(v_n)=\langle f,u_{\mathrm{cst}}\rangle u_\mathrm{cst}(n)+\langle f,u_\mathrm{alt}\rangle u_\mathrm{alt}(n)+\frac{4q}{(q^2-1)}\int_{0}^{\pi}\langle f,E_{\v}(\cdot,\frac{1}{2}+i\frac{\theta}{\log q})\rangle E_{\v}(v_n,\frac{1}{2}+i\frac{\theta}{\log q})d\theta.\]
\end{theorem}

\begin{proof}
This is a restatement of the main result in \cite{efrat}. One should keep in mind that the reference works with the measure $\frac{2q}{q^2-1}\mu$. By Proposition 3.1 of \cite{efrat}, there exists a unique function $f_{\T,\lambda}$ on $\Gamma\backslash\mathcal{V}(\Tcali)$ that is an eigenfunction for $\T$ with eigenvalue $\lambda$ and $f_{\T,\lambda}(v_0)=q+1$. If we set $\widehat{f_{\T,\theta}}=i\sin(\theta)f_{\T,2\sqrt{q}\cos(\theta)}$, Theorem 5.3 of \cite{efrat} yields the spectral resolution \[f(v_n)=\langle f,u_{\mathrm{cte}}\rangle u_{\mathrm{cte}}(v_n)+\langle f,u_{\mathrm{alt}}\rangle u_{\mathrm{alt}}(v_n)+\frac{4q}{q^2-1}\int_{0}^{\pi}\langle f,\widehat{f_{\T,\theta}}\rangle \widehat{f_{\T,\theta}}(v_n)\frac{d\theta}{(q-1)^2+4q\sin^2(\theta)}.\] At $s=\frac{1}{2}+\frac{\theta}{\log q}i$, we have $q^{1-s}=\sqrt{q}e^{-i\theta}$, $q^{1-2s}=e^{-2i\theta}$ and $q^s+q^{1-s}=2\sqrt{q}\cos(\theta)$. Then Lemma \ref{Eisenatzero} implies that 
\begin{equation}\label{eqfandEisen}\widehat{f_{\T,\theta}}(v_n)=\frac{i\sin(\theta)(1-qe^{-2i\theta})}{(1-e^{-2i\theta})} E_{\v}(v_n,\frac{1}{2}+\frac{\theta}{\log q}i).
\end{equation}
Using that $\frac{(1-qe^{-2i\theta})}{(1-e^{-2i\theta})}=(q+1)+i(q-1)\cot(\theta)$ and the trigonometric identity $(q+1)^2\sin^2(\theta)+(q-1)^2\cos^2(\theta)=(q-1)^2+4q\sin^2(\theta)$ one obtains the result.
\end{proof}

In view of the preceding theorem, in order to prove Theorems \ref{thmA'}, it suffices to establish the limit \eqref{testA} in the cases when $f$ is replaced by $u_{\mathrm{cte}}$, $u_{\mathrm{alt}}$ and $E_{\v}(\cdot,\frac{1}{2}+i\frac{\theta}{\log q})$, the first case being trivial. 

For a function $f$ as in Theorem \ref{thmB'}, the spectral decomposition in Theorem \ref{spectraldec} reads \[f(v_n)=2\int_{\Gamma\backslash\mathcal{V}(\mathcal{T})}fd\mu+\frac{4q}{(q^2-1)}\int_{0}^{\pi}\langle f,E_{\v}(\cdot,\frac{1}{2}+i\frac{\theta}{\log q})\rangle E_{\v}(v_n,\frac{1}{2}+i\frac{\theta}{\log q})d\theta,\] for $v_n$ in the support of $f$. Therefore, assuming the first claim in Theorem \ref{thmB} (see Proposition \ref{altaverage} below), Theorem \ref{thmB'} will be a consequence of establishing the limit \eqref{testB} with $f$ equal to $E_{\v}(\cdot,\frac{1}{2}+i\frac{\theta}{\log q})$.

We will immediately deal with the case $f=u_{\mathrm{alt}}$, while the treatment of Eisenstein series is postponed until Section \ref{sectionfundamentaldiscriminant}, after we have studied the spectral decomposition on edges in Section \ref{S spectral decomposition edges}.

\subsection{The case of the alternating function}\label{S alternatingfn}

A direct computation shows that $\int u_{\mathrm{alt}}(v_n)d{\mu}(v_n)=0$. Then, \eqref{testA} for the case $f=u_{\mathrm{alt}}$ is a consequence of the following proposition, which also proves the first claim in Theorem \ref{thmB}.

\begin{proposition}\label{altaverage} For $D$ as in Theorem \ref{thmA}, \[\#\{\phi\in CM(D):\mathrm{red}(\phi)=v_n\text{ with }n \text{ even}\}=\#\{\phi\in CM(D):\mathrm{red}(\phi)=v_n\text{ with }n \text{ odd}\}.\] In particular, $\frac{1}{h(D)}\sum_{\phi\in CM_D}u_{\mathrm{alt}}(\mathrm{red}(\phi))=0$. For $m\geq0$ and $D$ as in Theorem \ref{thmB}, the condition $\mathrm{CM}(D)\cap\mathrm{red}^{-1}(C_m)\neq\emptyset$ implies that $m$ and $\deg f$ have the same parity.
\end{proposition}

\begin{proof} Let $\mathfrak{a}$ be a proper $\mathcal{O}_D$-ideal so that $\phi_{z_\afrak}\in \mathrm{CM}(D)$. By Lemma \ref{discriminant}, we may assume that $az^2+bz+c=0$ with $b^2-4ac=D$ and $z\in \mathcal{F}$. By Lemma \ref{imaginary1} we have $|z|_i=|D|^{1/2}|a|^{-1}$. Then Lemma \ref{cmred} implies that $\mathrm{red}(\phi_{z_\mathfrak{a}})=v_n$ with $n=(\deg D)/2-\deg a$. Recall from the proof of Lemma \ref{discriminant} that $\mathcal{O}_D=A[az]$, from which it follows that $q^{\deg(a\mathfrak{a})}=|A/aA|=q^{\deg a}$ and so $\deg(\mathfrak{a})=-\deg(a)$.

Assume that $K_D=\F_{q^2}(T)$ and recall that $\mathcal{O}_{\F_{q^2}(T)}=\F_{q^2}[T]$. Since for any proper $\mathcal{O}_D$-ideal $\mathfrak{a}$ one has that $\mathcal{O}_D/\mathfrak{a}\cong \F_{q^2}[T]/\mathfrak{a}\F_{q^2}[T]$ is a $\F_{q^2}$-vector space, we conclude that $\deg(\mathfrak{a})$ is always even. By the conclusion of the previous paragraph, if an element in $\mathrm{CM}(D)$ reduces to a vertex $v_n$, $n$ must have the same parity as $(\deg D)/2=\deg f$.

Now assume that $K_D\neq\F_{q^2}(T)$. We first focus on $D$ square-free, in which case we have $\mathcal{O}_D=\mathcal{O}_K$. From our assumptions on $D$, we have the short exact sequence \begin{equation}\label{exactseq}0\to \mathrm{Cl}^0(K)\to \mathrm{Cl}(D)\xrightarrow{\deg} \Z/2\Z\to 0,\end{equation} where $\mathrm{Cl}^0(K)$ denotes the group of divisor classes of degree zero of the function field $K$. This is Proposition 14.1(b) in \cite{rosen} as explained in the first two cases of Proposition 14.7 \textit{loc. cit}.

By the conclusion of the first paragraph in the proof and \eqref{exactseq}, half of the classes in $\mathrm{Cl}(\mathcal{O}_D)$ have even degree, and half have odd degree. This implies the result for $D$ square-free. For general $D$, the surjective map $\mathrm{Cl}(\mathcal{O}_D)\to\mathrm{Cl}(\mathcal{O}_K)$ preserves the degree and we can argue similarly.
\end{proof}

\subsection{Spectral decomposition on $ L^2(\Gamma\backslash\mathcal{E}(\Tcali))$}\label{S spectral decomposition edges}
We say that two edges $e$ and $e'$ are \textit{adjacent} if they share a vertex and $e\neq e'$.

Let $\mathbb{U}$ be the operator on $\mathcal{E}(\Tcali)$ defined by \[\mathbb{U}(f)(e)=\sum_{e'\mbox{ is adjacent to e}}f(e').\] It descends to an operator on $\Gamma\backslash\mathcal{E}(\Tcali)$ given by \[\mathbb{U}(f)(e_n)=\begin{cases}(2q-1)f(e_1)+f(e_2),&\mbox{ if }n=1\\qf(e_{n-1})+(q-1)f(e_n)+f(e_{n+1}),&\mbox{ if }n\geq2.\end{cases}\]

\begin{lemma}
    The operator $\mathbb{U}$ is self-adjoint.
\end{lemma}

\begin{proof}
\begin{align*}
    (q-1)^{-1}\langle \mathbb{U}f,\overline{g}\rangle&=\sum_{n\geq1}\frac{\mathbb{U}f(e_n)g(e_n)}{q^n}\\
    &=\frac{((2q-1)f(e_1)+f(e_2))g(e_1)}{q}+\sum_{n\geq2}\frac{(qf(e_{n-1})+(q-1)f(e_n)+f(e_{n+1}))g(e_n)}{q^n}\\
    &=\frac{((2q-1)f(e_1)+f(e_2))g(e_1)}{q}+\frac{qf(e_1)g(e_2)}{q^2}+\frac{qf(e_2)g(e_3)}{q^3}+\frac{(q-1)f(e_2)g(e_2)}{q^2}\\&+\sum_{n\geq4}\frac{qf(e_{n-1})g(e_n)}{q^n}+\sum_{n\geq3}\frac{(q-1)f(e_n)g(e_n)}{q^n}+\sum_{n\geq2}\frac{f(e_{n+1})g(e_n)}{q^n}\\
    &=\frac{\mathbb{U}(g)(e_1)f(e_1)}{q}+\frac{\mathbb{U}(g)(e_2)f(e_2)}{q^2}+\sum_{n\geq3}\frac{(g(e_{n+1})+(q-1)g(e_n)+qg(e_{n-1}))f(e_n)}{q^n}\\
    &=(q-1)^{-1}\langle f, \mathbb{U}\overline{g}\rangle
\end{align*}
\end{proof}

\begin{proposition}\label{Ueigenfns}
Let $\lambda\in \R$ and $f$ such that $\mathbb{U}f=\lambda f$.
\begin{enumerate}
    \item If $\lambda\neq (q-1)\pm\sqrt{q}$, $f$ is a multiple of the function \[f_{\mathbb{U},\lambda}(e_n)=\frac{2q}{x_1-x_2}(x_1^{n-1}(x_1-q)-x_2^{n-1}(x_2-q)),\] where $x_1,x_2$ are the two different roots of $x^2-(\lambda-(q-1))x+q=0$.
    \item If $\lambda=(q-1)\pm\sqrt{q}$, $f$ is a multiple of the function \[f_{\mathbb{U},\lambda}(e_n)=2q(\pm\sqrt{q})^{n-2}(n(\pm\sqrt{q}-q)+q).\]
\end{enumerate}

\end{proposition}
\begin{proof}
The system $Uf=\lambda f$ give rise to a recurrence $u(n)=A^{n-1}u(1)$ where $u(n)=\begin{pmatrix}
    f(e_{n+1})\\f(e_n)
\end{pmatrix}$ and $A=\begin{pmatrix}
    \lambda-(q-1)&-q\\1&0
\end{pmatrix}$. If $f(e_1)=0$, $f\equiv0$ so we normalize $f(e_1)=2q$. The characteristic polynomial of $A$ is $x^2-(\lambda-(q-1))x+q=0$. Then, for $\lambda$ as in (1),  $A$ is diagonalizable. More explicitly $A=P\begin{pmatrix}
    x_1&0\\0&x_2
\end{pmatrix}P^{-1}$ with $P=\begin{pmatrix}
    x_1&x_2\\1&1
\end{pmatrix}$. From here, $f(e_n)$ is the second coordinate of \[\frac{1}{x_1-x_2}\begin{pmatrix}
    x_1 &x_2\\1&1
\end{pmatrix}\begin{pmatrix}
    x_1^{n-1} &0\\0&x_2^{n-1}
\end{pmatrix}\begin{pmatrix}
    1 &-x_2\\-1&x_1
\end{pmatrix}\begin{pmatrix}
    (\lambda-(2q-1))2q \\2q
\end{pmatrix}\]

which gives $f(e_n)=\frac{2q}{x_1-x_2}(((\lambda-(2q-1))(x_1^{n-1}-x_2^{n-1})-q(x_1^{n-2}-x_2^{n-2}))$. Using that $(\lambda-(2q-1))x_1-q=x_1^2-qx_1$ we get to the expression in (1).

In the case $\lambda=(q-1)\pm2\sqrt{q}$ we have a unique solution $x=\frac{\lambda-(q-1)}{2}$ and $A=PJP^{-1}$ with $J=\begin{pmatrix}
    x&1\\0&x
\end{pmatrix}$ and $P=\begin{pmatrix}
    x&1\\1&0
\end{pmatrix}$. Thus, $f(e_n)$ is the second coordinate of \[\begin{pmatrix}
    x&1\\1&0
\end{pmatrix}\begin{pmatrix}
    x^{n-1}&(n-1)x^{n-2}\\0&x^{n-1}
\end{pmatrix}\begin{pmatrix}
    0&1\\1&-x
\end{pmatrix}\begin{pmatrix}
    (\lambda-(2q-1))2q\\2q
\end{pmatrix}\] which equals $2q((2x-q)(n-1)x^{n-2}+(2-n)x^{n-1})=2qx^{n-2}(n(x-q)+q)$. So \[f(e_n)=2q(\pm\sqrt{q})^{n-2}(n(\pm\sqrt{q}-q)+q)\]
\end{proof}

\begin{proposition}
For $\lambda\in\R$, the function $f_{\mathbb{U},\lambda}$ is in $ L^2(\Gamma\backslash\mathcal{E}(\Tcali))$ if and only if $\lambda=2q$ in which case $f_\lambda$ is constant. 
\end{proposition}\begin{proof}
    Observe that for $f\colon \Gamma\backslash\mathcal{E}(\Tcali)\to\C$ to be square-integrable, it is necessary that $f(e_n)=o(q^{n/2})$. Suppose that $|q-1-\lambda|>2\sqrt{q}$ in which case $x_1$ and $x_2$ are real. Without loss of generality we assume $|x_1|>\sqrt{q}>|x_2|$. Since $x_2^{n-1}(x_2-q)=o(q^{n/2})$ and $|x_1|>q^{1/2}$, we must have $x_1=q$ for $f_{\mathbb{U},\lambda}(e_n)$ to be $o(q^{n/2})$. This implies that $x_2=1$, $\lambda-(q-1)=q+1$ and then $\lambda=2q$. The same reasoning shows that for $\lambda=(q-1)\pm2\sqrt{q}$, $f_{\mathbb{U},\lambda}$ is also not square-integrable.
    
    Finally, we assume that $|q-1-\lambda|<2\sqrt{q}$. This time, we have complex roots and we can write $x_2=\overline{x_1}$. Since $x_1x_2=q$ we can assume $x_1=\sqrt{q}e^{i\theta}$ for some $0<\theta<\pi$. Then \begin{align*}\frac{x_1-x_2}{4q}f_{\mathbb{U},\lambda}(e_n)=i\mathrm{Im}(x_1^{n-1}(x_1-q))&=iq^{\frac{n-1}{2}}\mathrm{Im}(e^{i(n-1)\theta}(\sqrt{q}e^{i\theta})-q)\\&=iq^{\frac{n-1}{2}}\mathrm{Im}(\sqrt{q}e^{in\theta}-qe^{i(n-1)\theta})\\&=iq^{\frac{n}{2}}(\sin(n\theta)-\sqrt{q}\sin((n-1)\theta))\end{align*}

The expression $\sqrt{q}\sin(n\theta)-q\sin((n-1)\theta)=\sqrt{q}\cos(n-1)\sin(\theta)+\sin((n-1)\theta)(\sqrt{q}\cos(\theta)-q)$ is the inner product between the constant vector $(\sqrt{q}\sin(\theta),\sqrt{q}\cos(\theta)-q)$ with $(\cos((n-1)\theta),\sin((n-1)\theta))$. Since the angle between these two vectors oscillates with $n$, we see that $(\sqrt{q}\sin(n\theta)-q\sin((n-1)\theta))$ does not converges to zero. In particular, $f_{\mathbb{U},\lambda}(e_n)$ is not $o(q^{n/2})$.
\end{proof}

For $0<\theta<\pi$, define \[\widehat{f}_{\mathbb{U},\theta}(e_n)=\frac{x_1-x_2}{4q}f_{\mathbb{U},(q-1)+2\sqrt{q}\cos(\theta)}=iq^{\frac{n}{2}}(\sin(n\theta)-\sqrt{q}\sin((n-1)\theta)).\] Extend the definition to $\theta\in[-\pi,\pi]$ by declaring $\widehat{f}_{\mathbb{U},-\theta}(e_n)=-\widehat{f}_{\mathbb{U},\theta}(e_n)$. 

Similarly, for $\psi\in L^2([0,\pi])$ we denote also by $\psi$ its odd extension to $[-\pi,\pi]$ i.e. $\psi(-\theta)=-\psi(\theta)$. Then \[\widehat{\psi}(n):=\frac{1}{2\pi}\int_{-\pi}^\pi\psi(\theta)e^{-in\theta}d\theta=\frac{-i}{2\pi}\int_{-\pi}^\pi\psi(\theta)\sin(n\theta)d\theta.\]

\begin{theorem}
For $\psi\in L^2[0,\pi]$, define \[F_\psi(e_n)=\frac{1}{2\pi}\int_{-\pi}^{\pi}\psi(\theta)\widehat{f}_{\mathbb{U},\theta}(e_n)d\theta.\] Then $F_\psi\in L^2(\Gamma\backslash\mathcal{E}(\Tcali))$ and \[\langle F_{\psi_1},F_{\psi_2}\rangle=\langle {\psi_1},{\psi_2}\rangle:=\frac{(q-1)}{2\pi}\int_0^\pi\psi_1(\theta)\overline{\psi_2(\theta)}(q+1-2\sqrt{q}\cos(\theta))d\theta\]
\end{theorem}
\begin{proof}
Observe that $\widehat{\psi}(0)=0$, $F_\psi(n)=q^{\frac{n}{2}}(\sqrt{q}\widehat{\psi}(n-1)-\widehat{\psi}(n))$ and $\overline{F_{\psi}}=-F_{\overline{\psi}}$. Then,
\begin{align*}
-(q-1)^{-1}\langle F_{\psi_1},F_{\overline{\psi_2}}\rangle&=(q-1)^{-1}\langle F_{\psi_1},\overline{F_{\psi_2}}\rangle\\
&=\sum_{n\geq1}(\sqrt{q}\widehat{\psi_1}(n-1)-\widehat{\psi_1}(n))(\sqrt{q}\widehat{\psi_2}(n-1)-\widehat{\psi_2}(n))\\
&=\sum_{n\geq1}q\widehat{\psi_1}(n-1)\widehat{\psi_2}(n-1)-\sqrt{q}(\widehat{\psi_1}(n-1)\widehat{\psi_2}(n)+\widehat{\psi_1}(n)\widehat{\psi_2}(n-1))+\widehat{\psi_1}(n)\widehat{\psi_2}(n)\\
&=\sum_{n\geq1}(q+1)\widehat{\psi_1}(n)\widehat{\psi_2}(n)-\sqrt{q}\widehat{\psi_1}(n)(\widehat{\psi_2}(n+1)+\widehat{\psi_2}(n-1))\end{align*}

Now we use the relation $\widehat{\psi}(n+1)+\widehat{\psi}(n-1)=2\widehat{(\psi\cos)}(n)$ and Parseval's formula $\sum_{n\geq1}\widehat{\psi}_1(n)\overline{\widehat{\psi}_2(n)}=\frac{1}{2\pi}\int_0^\pi{\psi}_1(\theta)\overline{{\psi}_2(\theta)}d\theta$ to obtain

\begin{align*}
-(q-1)^{-1}\langle F_{\psi_1},F_{\overline{\psi_2}}\rangle&=\sum_{n\geq1}\widehat{\psi_1}(n)((q+1)\widehat{\psi_2}(n)-2\sqrt{q}(\widehat{\psi_2\cos})(n))\\
&=-\frac{1}{2\pi}\int_0^\pi\psi_1(\theta)\psi_2(\theta)(q+1-2\sqrt{q}\cos(\theta))d\theta\\
&=-(q-1)^{-1}\langle \psi_1,\overline{\psi_2}\rangle
\end{align*}
\end{proof}

Denote by $dE$ the space of functions of the form $F_\psi$.

\begin{theorem}\label{spectraldecedges} We have $L^2(\Gamma\backslash\mathcal{E}(\mathcal{T}))=\C+dE$ and for every $f\in L^2(\Gamma\backslash\mathcal{E}(\mathcal{T}))$, \[f(e_n)=\int fd\nu+\frac{2}{(q-1)}\int_0^\pi\langle f,\widehat{f}_{\mathbb{U},\theta}\rangle \widehat{f}_{{\mathbb{U},\theta}}(e_n)\frac{d\theta}{(q+1-2\sqrt{q}\cos(\theta))}\]
\end{theorem}  
\begin{proof}
Indeed, if $f$ in $L^2(\Gamma\backslash\mathcal{E}(\mathcal{T}))$ is orthogonal to $F_\psi$ for of $\psi$, then \[0=\sum_{n\geq1}\frac{f(e_n)F_\psi(e_n)}{q^n}=\sum_{n\geq1}\frac{q^{n/2}f(e_n)(\sqrt{q}\widehat{\psi(n-1)}-\widehat{\psi}(n)}{q^n}=\sum_{n\geq1}q^{-n/2}(g(e_{n+1})-g(e_n))\widehat{\psi}(n)\]

viewing this equation as an equality of an inner product in $l^2(\N)$ with $\widehat{\psi}(n)$ varying, we conclude that $f(e_{n+1})=f(e_n)$ for all $n\geq1$ so $f$ is constant.

For the second claim, we argue using the density of $f\in dE\cap L^1(\Gamma\backslash\mathcal{E}(\mathcal{T}))$. Write $f=F_\psi$ and consider an interval $[\theta_0,\theta_0+h]\subseteq[0,\pi]$. On one hand, \[\lim_{h\to0}\langle f,F_{\mathbbm{1}_{[\theta_0,\theta_0+h]}}\rangle=\lim_{h\to0}\langle f,\frac{1}{2\pi}\int_{\theta_0}^{\theta_0+h}\widehat{f}_{\mathbb{U},\theta}(e_n)d\theta\rangle=\langle f,\widehat{f}_{\mathbb{U},\theta_0}\rangle.\] On the other hand, \[\lim_{h\to0}\langle f,F_{\mathbbm{1}_{[\theta_0,\theta_0+h]}}\rangle=\lim_{h\to0}\frac{(q-1)}{2\pi}\int_{\theta_0}^{\theta+h} \psi(\theta)(q+1-2\sqrt{q}\cos(\theta))d\theta=\frac{(q-1)}{2\pi}\psi(\theta)(q+1-2\sqrt{q}\cos(\theta))d\theta.\] We conclude that $\psi=\frac{2\pi}{(q-1)}\frac{\langle f,\widehat{f}_{\mathbb{U},\theta_0}\rangle}{(q+1-2\sqrt{q}\cos(\theta))}$

\end{proof}

\begin{remark}
In \cite{efrat}, the orthogonal complement of $\C u_{\mathrm{cst}}\oplus\C u_{\mathrm{alt}}$ is denote by $E$. Thus, the notation $dE$ is motivated by the following fact. If for a function $f$ on $\mathcal{V}(\Gamma\backslash\mathcal{T})$ we define $df(e_n):=f(v_{n-1})+f(v_n)$. One can compute that $d(\T f)=\mathbb{U}(df)-(q-1)df$. In particular, if $\T f=\tilde{\lambda}f$, we have that $\mathbb{U}(df)=(q-1+\tilde{\lambda})df$.
\end{remark}

Now we aim for the spectral decomposition to be in terms of Eisenstein series. From now on we denote $||\cdot||_{\boldsymbol{e}^+_{1/2}}$ simply by $||\cdot||_{\boldsymbol{e}}$. For $(x,y)\in k_\infty^2$ and $g\in \GL_2(k_\infty)$ we define $\psi_{\boldsymbol{e}}(g)=|\det(g)|||(0,1)g||_{\boldsymbol{e}}^{-2}$. This is a right $\mathbb{I}$-invariant function. Indeed, recall that \[||(x,y))||_{\boldsymbol{e}}=\sup\{||(x,y)||_{L_0},q^{1/2}||(x,y)||_{TL_1}\}.\] Then $\psi_{\boldsymbol{e}}(g)$ is right invariant by $\Gamma_0(T^{-1})$ which fixes the lattices $L_0$ and $TL_1$. Now, $w_\pi L_0=TL_1$ and $w_\pi TL_1=TL_0$ from where 
\begin{align*}
\psi_{\e}(gw_\pi)&=q^{-1}(\sup\{||(1,0)g||_{TL_1},q^{1/2}||(1,0)g||_{TL_0}\})^{-2}\\
&=q^{-1}(\sup\{||(1,0)g||_{TL_1},q^{-1/2}||(1,0)g||_{L_0}\})^{-2}\\
&=\psi_{\e}(g).\end{align*}

For $s\in\C$, we define the Eisenstein series \[E_{\boldsymbol{e}}(g,s)=\sum_{\gamma\in \Gamma_\infty\backslash \Gamma}\psi_{\boldsymbol{e}}(\gamma g)^s.\] For fixed $s$, the function $E_{\boldsymbol{e}}(g,s)$ is left $\Gamma$-invariant and right $\mathbb{I}$-invariant. Therefore, it defines a function $E_{\boldsymbol{e}}(e_{n+1},s):=E_{\boldsymbol{e}}(g_n,s)$ on $\Gamma\backslash \mathcal{E}(\Tcali)$. 

\begin{proposition}
The Eisenstein series satisfy \begin{equation}\label{eiseneigenvalueedge}\U(E_{\boldsymbol{e}}(\cdot,s))(g)=((q-1)+q^s+q^{1-s})E_{\boldsymbol{e}}(g,s).\end{equation}
\end{proposition}
\begin{proof}
It is enough to show that $\U(\psi_{\boldsymbol{e}}^s)(g)=((q-1)+q^s+q^{1-s})\psi_{\boldsymbol{e}}^s(g).$ To do this computation, we need control over the neighbors of an edge. Define  $s_\infty=\begin{pmatrix}T^{-1}&0\\0&1\end{pmatrix}$ and for $\alpha\in\F_q$ let $s_\alpha=\begin{pmatrix}T&0\\\alpha T&1\end{pmatrix}$. We start by the fact that the vertices in $\mathcal{V}(\Tcali)$ that are adjacent to a vertex of the form $g\boldsymbol{v}$ are the $(q+1)$ vertices $gs\boldsymbol{v}$ with $s\in\{s_\alpha\mid \alpha\in\mathbb{P}^1(\F_q)\}$. Note that $\boldsymbol{e}=\{\v,s_0\v\}$. Let $t_\infty=\begin{pmatrix}
    0&1\\1&0
\end{pmatrix}$ and $t_\alpha=\begin{pmatrix}1&0\\\alpha&1\end{pmatrix}$ for $\alpha\in \F_q$. Observe $t_\infty s_0=s_\infty\begin{pmatrix}
    0&T\\T&0
\end{pmatrix}$ and for $\alpha\in\F_q$, $t_\alpha s_0=s_\alpha$. Then, $g\e=\{g\v,gs_0\v\}$ and $gt_\alpha\e=\{g\v,gs_\alpha\v\}$ for $\alpha\in\mathbb{P}^{1}(\F_q)$. We conclude that the edges adjacent to $g\e$ agree with the collection of $2q$ edges formed by: the $q$ edges $gt_\alpha\e=\{g\v,gs_\alpha \v\}$ with $\alpha\in\F_q^{\times}\cup\{\infty\}$ and the $q$ edges $gs_0t_\alpha\e=\{gs_0\v,gs_0s_\alpha \v\}$ with $\alpha\in\F_q$. Write the second column of $g$ as $(x,y)$. Then \begin{align*}\mathbb{U}(\psi_{\e}^s)(g)&=\sum_{\alpha\in \F_q^\times\cup\{\infty\}}\psi_{\e}^s(gt_\alpha)+\sum_{\alpha\in \F_q}\psi_{\e}^s(gs_0 t_\alpha)\\&=|\det(g)|^s\left(\sum_{\alpha\in\F_q^\times}\sup\{q^{1/2}|x+\alpha y|,|y|\}^{-2s}+\sum_{\alpha\in \F_q}q^{s}\sup\{q^{1/2}|Tx+\alpha y|,|y|\}^{-2s}+\right.\\&\quad+\left.\sup\{q^{1/2}|y|,|x|\}^{-2s}\right).\end{align*}
At this point the verification is reduced to the study of the different possibilities for $|x|$ and $|y|$. We exemplify with the case $|x|=|y|$ and omit the rest of the cases. Suppose $|x|=|y|$ so that $\psi_{\e}^s(g)=|\deg(g)|^sq^{-s}|y|^{-2s}$. There exists a unique $\alpha\in\F_q^\times$ such that $|x+\alpha y|<|y|$. In the other cases, $|x+\alpha y|=|y|$. Then \begin{align*}\mathbb{U}(\psi_{\e}^s)(g)&=|\det(g)|^s(|y|^{-2s}+(q-2)q^{-2}|y|^{-2s}+q^{1-2s}|y|^{-2s}+q^{-s}|y|^{-2s})\\&=((q-1)+q^s+q^{1-s})\psi_{\e}^s(g)\end{align*}
\end{proof}

\begin{corollary}\label{comparisonfyeins}
At $s=\frac{1}{2}+\frac{i\theta}{\log q}$ we have \[\widehat{f}_{\mathbb{U},\theta}(e_n)=\frac{\sqrt{q}-qe^{-i\theta}}{2e^{-i\theta}}E_{\boldsymbol{e}}(e_n,s)\]
\end{corollary}

\begin{proof}
    By Proposition \ref{Ueigenfns} and Proposition \ref{eiseneigenvalueedge}, $E_{\e}(e_n,s)$ is a multiple of $\widehat{f}_{\mathbb{U},\theta}$ so it is enough to compare their value at $e_1$. We proceed as in Lemma \ref{Eisenatzero} and we skip some computations. From Lemma \ref{distinguishedred}, we need to compute \[E_{\boldsymbol{e}}(e_1,s)=\frac{1}{q-1}\sum_{(c,d)=A}\max\{q^{1/2}|c|,|d|\}^{-2s}.\] As before, it suffices to compute \[E^\star_{\boldsymbol{e}}(e_1,s):=\sum_{(0,0)\neq(c,d)}\max\{q^{1/2}|c|,|d|\}^{-2s},\] since $E^\ast_{\boldsymbol{e}}(e_1,s)=\zeta_{A}(2s)(q-1)E(e_1,s)$. We split the sum as \begin{equation}\label{sume1}
    E^\star_{\boldsymbol{e}}(e_1,s)=\sum_{0\neq d}|d|^{-2s}+\sum_{c\neq 0,d\in A}\max\{q^{1/2}|c|,|d|\}^{-2s}.\end{equation} The second summand in \eqref{sume1} becomes \begin{align}\label{sume12}\sum_{c\neq 0,d\in A}\max\{q^{1/2}|c|,|d|\}^{-2s}&=\sum_{n\geq0}(q-1)(q^{2-2s})^n\left(\frac{q^{1-s}(1-q^{1-2s})+(q-1)q^{1-2s}}{1-q^{1-2s}}\right)\nonumber\\&=\frac{(q-1)q^{1-s}(1-q^{1-2s})+(q-1)^2q^{1-2s}}{(1-q^{1-2s})(1-q^{2-2s})}.\end{align}
    
    Adding \eqref{sum11} and \eqref{sume12}, we obtain

\[E^\star_{\boldsymbol{e}}(e_1,s)=\frac{q-1}{1-q^{1-2s}}+\frac{(q-1)q^{1-s}(1-q^{1-2s})+(q-1)^2q^{1-2s}}{(1-q^{1-2s})(1-q^{2-2s})}=\frac{q-1}{1-q^{1-s}}.\]

In conclusion $E_{\boldsymbol{e}}(e_1,s)=\frac{(1-q^{1-2s})}{1-q^{1-s}}$ and we finish the proof by specializing at $s=\frac{1}{2}+\frac{i\theta}{\log q}$.
\end{proof}

\begin{theorem}\label{Thm spectral resolution on edges}
    The spectral resolutions for $f\in L^2(\Gamma\backslash\mathcal{E}(\mathcal{T}))$ reads \[f(e_n)=\int fd\nu+\frac{q}{2(q-1)}\int_0^\pi\left\langle f,E_{\e}(\cdot,\frac{1}{2}+\frac{i\theta}{\log q})\right\rangle E_{\e}(e_n,\frac{1}{2}+\frac{i\theta}{\log q})d\theta\]
\end{theorem}

\begin{proof}
This is a direct consequence of Theorem \ref{spectraldecedges} and Corollary \ref{comparisonfyeins}. Indeed, the corollary implies that \[\langle f,\widehat{f}_{\mathbb{U},\theta}\rangle\widehat{f}_{\mathbb{U},\theta}(e_n)=\left|\frac{\sqrt{q}-qe^{-i\theta}}{2e^{-i\theta}}\right|^2\left \langle f,E_{\e}(\cdot,\frac{1}{2}+\frac{i\theta}{\log q})\right\rangle E_{\e}(e_n,\frac{1}{2}+\frac{i\theta}{\log q})\] and one computes that $\left|\frac{\sqrt{q}-qe^{-i\theta}}{2e^{-i\theta}}\right|^2=\frac{q(q+1-2\sqrt{q}\cos(\theta))}{4}$.
\end{proof}

As a consequence of Theorem \ref{Thm spectral resolution on edges}, it follows that, as for the function over vertices, the proof of Theorem \ref{thmC'} is reduced to work the limit $\eqref{testC}$ with $f$ equal to $E_{\e}(\cdot,\frac{1}{2}+\frac{i\theta}{\log q})$.

\section{The case of Eisenstein series and fundamental discriminant}\label{sectionfundamentaldiscriminant}

In this section, we prove Theorems \ref{thmA'} and \ref{thmC'} in the case of a sequence of square-free discriminants $D$. As a consequence of the previous section, we need to prove the limits \ref{testA} and \ref{testC} for Eisenstein series. Note that for a sequence of discriminants $D$ as in Theorem \ref{thmB'}, the square-free part of $D$ is fixed. For the rest of this section, $D$ is square-free.

\begin{lemma}\label{integralfthetacero} The following identity holds:
    $$\int_{\Gamma\backslash \mathcal{V}(\Tcali)}E_{\boldsymbol{v}}\left(v,\dfrac{1}{2}+i\dfrac{\theta}{\log q} \right)d\mu(v)=0=\int_{\Gamma\backslash \mathcal{E}(\Tcali)}E_{\boldsymbol{e}}\left(e,\dfrac{1}{2}+i\dfrac{\theta}{\log q} \right)d\nu(e).$$
\end{lemma}

\begin{proof}
    We only prove the case over vertices. By \eqref{eqfandEisen}, it suffices to show that
    \[\int_{\Gamma \backslash\mathcal{V}(\Tcali)} \widehat{f}_{\mathbb{T},\theta}d\mu=0.\]
    An explicit description of the functions $\widehat{f}_{\mathbb{T},\theta}$ can be found in \cite{efrat}. More precisely, $\widehat{f}_{\mathbb{T},\theta}(v_0)=(q+1)i\sin(\theta)$, and for $n\geq 1$, $\widehat{f}_{\mathbb{T},\theta}(v_n)=q^{n/2}i\left(\sin((n+1)\theta)-q\sin((n-1)\theta) \right)$. Therefore,

    \begin{align*}
        \int_{\Gamma \backslash\mathcal{V}(\Tcali)} \widehat{f}_{\mathbb{T},\theta}d\mu&=i\sin(\theta)\mu(v_0)+\sum_{n\geq 1}q^{n/2}i\left(\sin((n+1)\theta)-q\sin((n-1)\theta) \right)\mu(v_n)\\
        &=\frac{q-1}{2q}(q+1)i\sin(\theta)+\sum_{n\geq 1}\frac{q^2-1}{2q^{n+1}}q^{n/2}i\sin\left((n+1)\theta \right)-\sum_{n\geq 1}\frac{q^2-1}{2q^{n+1}}q^{n/2}qi\sin((n-1)\theta)\\
        &=\frac{(q^2-1)i\sin(\theta)}{2q}-\frac{q^2-1}{2q^2}qi\sin(0)-\frac{q^2-1}{2q^3}q^2i\sin(\theta)+\\ &\ \ \ \ \ \ \ \ \ \ \ \ \ \ \ \ \ \ \ \ \ \ \ \ \ \ \ +\sum_{n\geq 1} \frac{q^2-1}{2q^{n+1}}q^{n/2}i\sin((n+1)\theta)-\sum_{n\geq 1}\frac{q^2-1}{2q^{n+3}}q^{n/2+2}i\sin((n+1)\theta)\\
        &=0.
    \end{align*}
\end{proof}

By Lemma \ref{integralfthetacero}, in order to prove Theorem \ref{thmA'} and \ref{thmC'}, it remains to show that
\begin{equation}\label{testEisen0}\dfrac{1}{h(D)}\sum_{[\afrak]\in \mathrm{Cl}(D)}E\left(\overline{\lambda}(z_{\afrak}),\dfrac{1}{2}+i\dfrac{\theta}{\log q} \right)\to 0 \text{ as } \deg D\to \infty. \end{equation}

To study this convergence, as in the number field setting, we express the average of Eisenstein series in terms of $L$-functions.

Let $g\in \GL_2(k_\infty)$ and write $g\cdot\i=z=x+y\i\in k_{\infty^2}\smallsetminus k_\infty$. Then $$g=\begin{pmatrix}
    y&x\\0&1
\end{pmatrix}k$$ for some $k\in \GL_2(k_\infty)$ with $k\i=\i$. After applying $\lambda$, we conclude that $k\in\mathbb{K}$. Write $\gamma=\begin{pmatrix}
    a&b\\c&d
\end{pmatrix}$, and then

\[|\gamma z|_i=\frac{|\det(\gamma)||z|_i}{|cz+d|^2}=\frac{|\det(\gamma)y|}{\max\{|cx+d|,|cy|\}^2}=\psi_{\v}(\gamma\begin{pmatrix}
    y&x\\0&1
\end{pmatrix})=\psi_{\v}(\gamma g),\] since $\psi_{\v}$ is right-invariant under $\mathbb{K}$.

Similarly, if $g\cdot\r=z=x+y\r\in k_{ram}\smallsetminus k_\infty$, $g=\begin{pmatrix}
    y&x\\0&1
\end{pmatrix}k'$ with $k'\in\mathbb{I}$ and \[|\gamma z|_i=\frac{|\det(\gamma)||z|_i}{|cz+d|^2}=\frac{q^{1/2}|\det(\gamma)y|}{\max\{q^{1/2}|cy|,|cx+d|\}^2}=q^{1/2}\psi_{\e}(\gamma\begin{pmatrix}
    y&x\\0&1
\end{pmatrix})=q^{1/2}\psi_{\e}(\gamma g),\] since $\psi_{\e}$ is right-invariant under $\mathbb{I}$.

Therefore, the Eisenstein series $E(z,s)$ on $\Gamma\backslash\Omega$, defined by \[E(z,s)=\sum_{\gamma\in\Gamma_\infty\backslash\Gamma}|\gamma z|_i^{s}=\sum_{(c,d)=A}\frac{|z|_i^s}{|cz+d|^{2s}}\] satisfies $E(z,s)=E_{\boldsymbol{v}}(\overline{\lambda}(z),s)$ for $z\in k_{\infty^2}\smallsetminus k_\infty$ and $E(z,s)=q^{s/2}E_{\boldsymbol{e}}(\overline{\lambda}(z),s)$ for $z\in k_{ram}\smallsetminus k_\infty$.

We conclude that \eqref{testEisen0} is equivalent to \begin{equation}\label{testEisen}\dfrac{1}{h(D)}\sum_{[\afrak]\in \mathrm{Cl}(D)}E\left(z_{\afrak},\dfrac{1}{2}+i\dfrac{\theta}{\log q} \right)\to 0 \text{ as } \deg D\to \infty. \end{equation}

Let $\chi_D$ be the quadratic Dirichlet character associated with $D$ i.e., $\chi_D(a)=\binom{D}{a}$ is the quadratic residue symbol. Its associated $L$-series is \[L(s,\chi_D)=\sum_{f\in A_+}\chi_D(f)|f|^{-s}.\]

For a character $\chi$ of the group $\mathrm{Cl}(D)$, recall that \[L_K(s,\chi)=\sum_{\mathfrak{a}\subseteq A[\sqrt{D}]}\chi(\mathfrak{a})N(\mathfrak{a})^{-s}.\]

If $\chi_0$ denotes the trivial class group character, we have the product formula \begin{equation}\label{trivialLseries}L_K(s,\chi_0)=\zeta_A(s)L(s,\chi_D),\end{equation} which follows by comparing Euler products.

\begin{lemma}\label{eisensteinsum} Let $s\neq 1+\frac{k\pi i}{\log q}$ for any $k\in\Z$. The following formula holds. \[\sum_{\mathfrak{a}\in \mathrm{Cl}(D)}E(z_\mathfrak{a},s)=\frac{\#\mathcal{O}_D^\times\zeta_A(s)}{(q-1)\zeta_A(2s)}|D|^{s/2}L(s,\chi_D).\]
\end{lemma}

\begin{proof}
We replicate the proof over number fields for which we refer to \cite[Chap. 22, Section 3]{iwanieckowalski}. Since $L_K(s,\chi)=L_K(s,\overline{\chi})$, orthogonality of characters yields \[\sum_{\chi\in \widehat{\mathrm{Cl}(\mathcal{O}_D)}}\chi(\mathfrak{a})L_K(s,\chi)=h(D)\sum_{\mathfrak{b}=\alpha\mathfrak{a}\subseteq\mathcal{O}_D}N(\mathfrak{b})^{-s}=\frac{h(D)}{\#\mathcal{O}_D^\times}|a|^{-s}\sum_{0\neq\alpha\in \mathfrak{a}^{-1}}N(\alpha)^s.\]

Since $\mathfrak{a}^{-1}=A+\overline{z_\afrak}A$ (where $\overline{z_\afrak}$ denotes the Galois conjugate of $z_\afrak$), we have \[\sum_{0\neq\alpha\in \mathfrak{a}^{-1}}N(\alpha)^s=\sum_{(0,0)\neq (c,d)\in A^2} |cz_\afrak+d|^{-2s}=(q-1)\zeta_A(2s)|z_\afrak|_i^{-s}E(z_\mathfrak{a},s).\]

Putting the above together, \[\sum_{\chi\in \widehat{\mathrm{Cl}(\mathcal{O}_D)}}\chi(\mathfrak{a})L_K(s,\chi)=\frac{(q-1)}{\#\mathcal{O}_D^\times}|D|^{-s/2}\zeta(2s)E(z_\mathfrak{a},s),\] which by Fourier inversion implies that \[\sum_{\mathfrak{a}\in \mathrm{Cl}(\mathcal{O}_D)}\chi(\mathfrak{a})E(z_\mathfrak{a},s)=\frac{\#\mathcal{O}_D^\times}{(q-1)\zeta_A(2s)}|D|^{s/2}L_K(s,\chi).\] The result follows by taking $\chi=\chi_0$ and \eqref{trivialLseries}.
\end{proof}






By Lemma \ref{eisensteinsum} we obtain that \[\left|\sum_{\mathfrak{a}\in \mathrm{Cl}(D)}E(z_\mathfrak{a},\frac{1}{2}+it)\right|\ll |D|^{1/4}\left|L(\frac{1}{2}+it,\chi_D)\right|.\]

Moreover, $|h(D)|\gg_\varepsilon|D|^{1/2-\varepsilon}$ (see \cite{papanikolasequidistribution} equation (8) on p.5) and $|L(\frac{1}{2}+it,\chi_D)|\ll_\varepsilon|D|^\varepsilon$ (see \cite[Theorem 5.1]{bucur2018traces} and its extension \cite[Theorem A1]{Diaconu}). We therefore conclude that 

\begin{equation}\label{eisencase}
\left|\frac{1}{h(D)}\sum_{\mathfrak{a}\in \mathrm{Cl}(D)}E(z_\mathfrak{a},\frac{1}{2}+it)\right|\ll |D|^{-1/4+\varepsilon}.\end{equation}

\section{General discriminant}\label{sectiongeneraldiscriminant}

In this section, we explain how to extend Theorems \ref{thmA} and \ref{thmB}, from square-free discriminant to general discriminants. The argument also gives the proof for Theorem \ref{thmB'}. We do so by following the strategy presented in \cite{clozelullmo}. We begin by introducing Hecke operators.

\subsection{Hecke operators}

Let $\nfrak$ be a non-zero ideal of $A$, generated by the monic polynomial $n\in A$. We define the Hecke correspondence on $\Gamma\backslash\Omega$, defined by

\[T_\nfrak  z=\sum_{\substack{a \text{ monic, } ad=n\\\deg b<\deg d}}\dfrac{az+b}{d}.\]

This correspondence induces a linear operator on complex-valued functions $f\colon \Gamma \setminus \Omega \to \C$ by \[T_\nfrak f(z)=\sum_{y\in T_\nfrak z}f(y).\] We prove that Eisenstein series are eigenfunctions for all Hecke operators. Before doing so, we first state some of their properties.

\begin{theorem}\label{propertiesheckeoperator}
The Hecke operators satisfy the following properties:
\begin{itemize}
    \item[(i)] If $\mfrak$ and $\nfrak$ are coprime ideals in $A$,  then $T_\mfrak T_\nfrak=T_\nfrak T_\mfrak$. 
    \item[(ii)] Let $\pfrak$ be a prime ideal of $A$. Then, for all $n\geq 1,$
    $$T_{\pfrak}T_{\pfrak^n}=T_{\pfrak^{n+1}}+|\pfrak|T_{\pfrak^{n-1}}.$$
\end{itemize}
\end{theorem}

\begin{proof}\noindent 
\begin{itemize}
    \item[(i)] It suffices to show that  $T_{\mfrak \nfrak}=T_{\mfrak}T_{\nfrak}$, with $\mfrak$ and $\nfrak$ coprime ideals of $A.$ Let $m$ and $n$ be monic generators of $\mfrak$ and $\nfrak$, respectively.
\begin{align*}
    T_{\mfrak}T_{\nfrak}&=T_{\mfrak}\left(\sum_{\substack{a\text{ monic, } ad=n\\ \deg b<\deg d}} f\left( \frac{az+b}{d}\right)\right)\\
    &=\sum_{\substack{\alpha\text{ monic, } \alpha\delta=m\\ \deg \beta<\deg \delta}}\sum_{\substack{a\text{ monic, } ad=n\\ \deg b<\deg d}}f\left( \dfrac{a\alpha z+a\beta+b\delta}{\delta d}\right).
\end{align*}

Since $a$ and $\delta$ are coprime, and $\deg b<\deg n$, $\deg \beta < \deg m$, the sum runs over all elements of degree smaller than 
$\deg mn$. In this way, the last sum is equal to
$$\sum_{\substack{a\text{ monic, } ad=mn\\ \deg b<\deg d}} f\left( \frac{az+b}{d}\right).$$
We conclude that $T_{\mfrak \nfrak}=T_{\mfrak}T_{\nfrak}.$
    \item[(ii)] Let $p$ be a monic generator of $\pfrak$. Then 
    \begin{align*}
        T_\pfrak T_{\pfrak^n}f(z)&= T_\pfrak\left(\sum_{k=0}^{n}\sum_{ \deg b< \deg p^{k}} f\left( \dfrac{p^{n-k}z+b}{p^k}\right)\right)\\
        &=\sum_{k=0}^{n}\sum_{ \deg b< \deg p^{k}}\left( f\left( \dfrac{p^{n-k+1}z+b}{p^k}\right)+\sum_{\deg Y< \deg p}f\left(\dfrac{p^{n-k}\left( \frac{z+Y}{p}\right)+b}{p^k} \right)\right)\\
        &=\sum_{k=0}^n \sum_{ \deg b < \deg p^k}f\left(\dfrac{p^{n-k+1}z+b}{p^k} \right)+\sum_{ \deg b < \deg p^n}\sum_{\deg Y < \deg p}f\left(\frac{z+Y+p b}{p^{n+1}} \right)\\
    &     \ \ \ \ \ \ \ \ \ \      +  \sum_{k=0}^{n-1}\sum_{\deg b<\deg  p^k}\sum_{\deg Y< \deg p}f\left(\dfrac{p^{n-k}z+p^{n-k}Y+b}{p^k} \right)\\
    &=\sum_{k=0}^{n+1}\sum_{\deg b< \deg p^{k}}f\left(\frac{p^{n+1-k}z+b}{p^k} \right)+|\pfrak|\sum_{k=0}^{n-1}\sum_{\deg b< \deg p^{k}}f\left(\frac{p^{n-1-k}z+b}{p^k} \right)\\
    &=T_{\pfrak^{n+1}}+|\pfrak|T_{\pfrak^{n-1}}.
\end{align*}
\end{itemize}    
\end{proof}


To describe the eigenvalues of Eisenstein series, we need the following arithmetic function.

\begin{definition}\label{definitionfunctionsigma}
    Let $s\in \C$. The arithmetic function $\sigma_s\colon A_+\to \C$ is defined by
    \[\sigma_s(n)=\sum_{\substack{m \text{ monic}\\ m|n}}|m|^s.\]
\end{definition}

\begin{remark}\label{multiplicativesigma}
Similar to the classical $\sigma$-function, this is multiplicative; that is, $\sigma_s(m n)=\sigma_s(m)\sigma_s(n)$ if $m$ and $n$ are coprime. If $ \mathfrak{n}=nA$ with $n\in A_+$, we set $\sigma_s(\mathfrak{n}):=\sigma_s(n)$.
\end{remark}

\begin{lemma}\label{eigenfunctiontp}
    Let $\pfrak$ be a prime ideal of $A$. Then,
    \[T_\pfrak E(z,s)=|\pfrak|^s\left(1+|\pfrak|^{1-2s} \right)E(z,s).\]
\end{lemma}
\begin{proof} Let $p$ be a monic generator of $\pfrak$. The Hecke operator $T_\pfrak$ acts as follows:
    \[T_\pfrak E(z,s)=E(pz,s)+\sum_{\deg b<\deg p}E\left(\frac{z+b}{p},s\right).\]

    After a lengthy computation, one obtains 
\begin{equation}\label{sump}
E(pz,s)=|z|_i^s|p|^{-s}\sum_{\substack{(c,d)=A\\p\nmid c}}|cz+d|^{-2s}+|z|_i^s|p|^s\sum_{\substack{(c,d)=A\\p\mid c}}|cz+d|^{-2s},
\end{equation}

and
\begin{equation}\label{sumoverb}
    \sum_{\deg b<\deg p}E\left(\frac{z+b}{p},s\right)=|z|_i^s|p|^s\sum_{\substack{(c,d)=A\\ p\nmid c}}{\left|cz+d\right|^{-2s}}+|p|^{1-s}E(z,s)-|z|_i^s|p|^{-s}\sum_{\substack{(c,d)=A\\p\nmid c}}{\left|cz+d\right|^{2s}}.
\end{equation}
    
Adding \eqref{sump} and \eqref{sumoverb} one obtains \[T_\pfrak E(z,s)=(|\pfrak|^s+|\pfrak|^{1-s})E(z,s).\]

\end{proof}

\begin{theorem}\label{eigenfunctiontn}
    Let $\nfrak$ be a non-zero ideal of $A$ generated by a monic polynomial $n$. Then
    $$T_\nfrak E(z,s)=|\nfrak|^s \sigma_{1-2s}(\nfrak)E(z,s).$$
\end{theorem}

\begin{proof}
    By Theorem \ref{propertiesheckeoperator}(i) and the multiplicativity of $\sigma_{1-2s}$, it suffices to treat the case $\nfrak=\pfrak^n$.
    We show it by induction.
    If $n=1$, then the result follows from Lemma \ref{eigenfunctiontp}.
    We suppose that the result holds for $T_{\pfrak^n}$. 
    Using recurrence formula of $T_{\pfrak^{n+1}}$ (Theorem \ref{propertiesheckeoperator}(ii)) and hypotheses of induction, we get
    \begin{align*}
        T_{\pfrak^{n+1}}E(z,s)&=T_\pfrak T_{\pfrak^{n}}E(z,s)-|\pfrak|T_{\pfrak^{n-1}}E(z,s)\\
        &=\left(\left(|\pfrak|^s+|\pfrak|^{1-s} \right)|\pfrak^n|^s\sum_{k=0}^n |\pfrak^k|^{1-2s}-|\pfrak||\pfrak^{n-1}|^s\sum_{k=0}^{n-1}|\pfrak^k|^{1-2s}\right)E(z,s)\\
        &=\left(|\pfrak^{n+1}|^s\sum_{k=0}^n |\pfrak^k|^{1-2s}+|\pfrak^{1-s+ns}|\sum_{k=0}^n|\pfrak^k|^{1-2s}-|\pfrak^{1+ns-s}|\sum_{k=0}^{n-1}|\pfrak^k|^{1-2s}\right)E(z,s)\\
        &=|\pfrak^{n+1}|^s\sum_{k=0}^{n+1}|\pfrak^k|^{1-2s}E(z,s)\\
        &=|\pfrak^{n+1}|^s\sigma_{1-2s}(\pfrak^{n+1})E(z,s).
    \end{align*}
\end{proof}

\subsection{The case of general discriminant}

Let $K/k$ be an imaginary quadratic extension contained in $\C_\infty$. Recall that every $A$-order $\mathcal{O}$ in $K$ is determined by its conductor, that is, the unique $A$-ideal $\mathfrak{f}$ for which $\mathcal{O}=A+\mathfrak{f}\mathcal{O}_K$. In what follows, we denote by $\mathcal{O}_{K,\mathfrak{f}}$ the order in $K$ of conductor $\mathfrak{f}$. Fix a square-free $D_K$ such that $\mathcal{O}_K=A[\sqrt{D_K}]$, and suppose that $\mathfrak{f}=fA$. Then $\mathcal{O}_{K,\mathfrak{f}}=A[f\sqrt{D_K}]$. We define the quantities $h(K,\mathfrak{f})$ and $\mathrm{CM}(K,\mathfrak{f})$ accordingly. Observe that  $\mathrm{CM}(K,\mathfrak{f})=\mathrm{CM}(D)$ with $D=D_Kf^2$.

We use the notation $\mathrm{Div}(K,\mathfrak{f})$ to denote the divisor on $ Y(\mathbb{C}_\infty)$ given by \[\sum_{\phi\in \mathrm{CM}(\mathcal{O}_{K,\mathfrak{f}})}\phi.\]

For an $A$-ideal $\mathfrak{n}$, let $R_K(\mathfrak{n})$ be the number of integral ideals in $\mathcal{O}_K$ with norm $\mathfrak{n}$. Set $w_{K,\mathfrak{f}}=[\mathcal{O}_{K,\mathfrak{f}}^\times\colon \mathbb{F}_q^\times]$. This quantity equals $q+1$ in the case when $K=\F_{q^2}(T)$ and $\mathfrak{f}=A$; otherwise, it assume the value $1$. 

\begin{lemma}\label{Zhang}
We have the following equality of divisors on $Y(\C_\infty)$:
\[T_\ffrak\left(\frac{1}{w_{K,A}}\mathrm{Div}({K},A)\right)=\sum_{\mathfrak{c}\mid\mathfrak{f}}R_K\left(\mathfrak{f}\mathfrak{c}^{-1}\right)\frac{1}{w_{K,\mathfrak{c}}}\mathrm{Div}({K,\mathfrak{c}}).\]
\end{lemma}

This is the function field analogue of Zhang's Proposition 4.2.1 \cite{Zhang}. The proof follows the same lines, and to reproduce it, we recall the idelic description of the class group. For a prime $\mathfrak{p}$ of $A$, $A_\mathfrak{p}$ denotes the completion of $A$ at $\mathfrak{p}$. The ring of finite adeles $\prod_\mathfrak{p}A_\mathfrak{p}$ is denoted by $\widehat{A}$ and $\widehat{k}$ (resp. $\widehat{K}$) denotes the ring of finite adeles $k\otimes\widehat{A}$ (resp. $K\otimes\widehat{A}$). For every order $\mathcal{O}$ in $K$, the map sending $x\in \widehat{K}^\times$ to $K\cap x\widehat{\mathcal{O}}$ induces an isomorphism \[K^\times\backslash\widehat{K}/\widehat{\mathcal{O}}^\times\cong \mathrm{Cl}(\mathcal{O}).\]

\begin{proof}[Proof of Lemma \ref{Zhang}] By Lemma 3.5 in \cite{Goss80}, the sum defining $T_\mathfrak{f}z$ can be seen as a sum over the homothety classes of lattices admitting a representative $L\subseteq\Lambda_z$ of index $\mathfrak{f}$. 
Let $\mathfrak{a}$ be a proper $\mathcal{O}_{K,\mathfrak{c}}$-ideal and $\mathfrak{b}$ a $\mathcal{O}_K$-ideal, and suppose that their images in the class group correspond to $x$ and $y$ in $\widehat{K}^\times$ respectively. Then, for $\lambda\in K^\times$, $\lambda \mathfrak{a}\subseteq\mathfrak{b}$ with index $\mathfrak{f}$ if and only if $\lambda x\widehat{\mathcal{O}}_{K,\mathfrak{c}}$ is a sublattice of $ y\widehat{\mathcal{O}}_{K}$ of index $\mathfrak{f}$. Equivalently, $y^{-1}\lambda x\in \widehat{\mathcal{O}}_{K}$ and the norm of $y^{-1}\lambda x$ is $\mathfrak{f}/\mathfrak{c}$.

Let $S$ be a set of representatives for $K^\times\backslash\widehat{K}/\widehat{\mathcal{O}}_K^\times$ and consider the map $S\times K^\times/\mathcal{O}_{K,\mathfrak{f}}^\times\to \widehat{K}^\times/\widehat{\mathcal{O}}_{K}^\times$ induced by $(y,\lambda)\mapsto y^{-1}\lambda x$. Suppose that $y_1^{-1}\lambda_1 x=y^{-1}_2\lambda_2 xu$ with $u\in\widehat{\mathcal{O}}_K$. Then $y_2=\lambda_1^{-1}\lambda_2y_1u$, from which it follows that $y_1=y_2$ by the definition of $S$ and so $\lambda_1^{-1}\lambda_2\in\mathcal{O}_K^\times=K^\times\cap\widehat{\mathcal{O}}_K^\times$. We conclude that the map is $[\mathcal{O}_K^\times\colon \mathcal{O}_{K,\mathfrak{f}}^\times]$-to-one. Therefore, making $\gamma:=y^{-1}\lambda x$ we get that the multiplicity of $\Phi_\mathfrak{a}$ in $T_\mathfrak{f}\left(\mathrm{Div}(K,A)\right)$ is \[\frac{1}{w_{K,A}}[\mathcal{O}_K^\times\colon \mathcal{O}_{K,\mathfrak{f}}^\times]\#\{\gamma\in \widehat{\mathcal{O}}_K/\widehat{\mathcal{O}}_K^\times\mid N(\gamma)=\mathfrak{f}\mathfrak{c}^{-1}\}=\frac{1}{w_{K,\mathfrak{c}}}R_K(\mathfrak{f}\mathfrak{c}^{-1}).\]
\end{proof}

Let $\boldsymbol{1}$ the constant arithmetic function equal to $1$. The identity \eqref{trivialLseries} implies that $R_K=\chi_{D_K}*\boldsymbol{1}$ where $*$ indicates Dirichlet convolution. Letting $R_K^{-1}$ denote the inverse of $R_K$ under $*$, we conclude that \begin{equation}\label{boundRK}
|R_K^{-1}(\mathfrak{n})|\ll_\varepsilon|\mathfrak{n}|^\varepsilon.
\end{equation} 

Assume $\deg(D)\geq1$ so that $w_{K,\mathfrak{f}}=1$. By Lemma \ref{Zhang} and M{\"o}bius inversion formula, we have

\begin{align*}
\sum_{\mathfrak{a}\in \mathrm{Cl}(\mathcal{O}_{K,f})}E(z_\mathfrak{a},s)&=\sum_{\mathfrak{a}\in\mathrm{Cl}(\mathcal{O}_K)}\sum_{\mathfrak{c}\mid \mathfrak{f}}R^{-1}_K(\mathfrak{f}\mathfrak{c}^{-1})(T_{\mathfrak{c}}E(z_\mathfrak{a},s))\\&=\sum_{\mathfrak{c}\mid \mathfrak{f}}R^{-1}_K(\mathfrak{f}\mathfrak{c}^{-1})|\mathfrak{c}|^s\sigma_{1-2s}(\mathfrak{c})\sum_{\mathfrak{a}\in \mathrm{Cl}(\mathcal{O}_{K})}E(z_\mathfrak{a},s).
\end{align*}

On the other hand, the formula (Proposition 17.9 in \cite{rosen}) \[h(K,\mathfrak{f})=h(K)\frac{|\mathfrak{f}|}{[\mathcal{O}_K^\times :\mathcal{O}_{K,\mathfrak{f}}^\times]}\prod_{\mathfrak{p}\mid\mathfrak{f}}\left(1-\frac{\chi_K(\mathfrak{p})}{|\mathfrak{p}|}\right)\] implies that  $h(K,\mathfrak{f})\gg_\varepsilon h(K)|\mathfrak{f}|^{1-\varepsilon}$. From the square-free case \eqref{eisencase}, together with \eqref{boundRK} and $|\sigma_s(\mathfrak{n})|\ll_\varepsilon|\mathfrak{n}|^{\varepsilon+\mathrm{Re}(s)}$, we conclude, with $s=1/2+\frac{\theta}{\log q}i$, that

\[\left|\frac{1}{h(K,\mathfrak{f})}\sum_{\mathfrak{a}\in \mathrm{Cl}(\mathcal{O}_{K,\mathfrak{f}})}E(z_\mathfrak{a},s)\right|\ll_\varepsilon\frac{|D_K|^{-1/4+\varepsilon}}{|\mathfrak{f}|^{1/2-\varepsilon}}=|D|^{-1/4+\varepsilon}.\] This proves Theorem \ref{thmB'} and complete the proof of Theorems \ref{thmA'} and \ref{thmC'}.

\section*{Acknowledgements} 
M.A. was partially supported by ANID Fondecyt  postdoctoral grant 3261344 from Chile. \\ We thank Fabien Pazuki for useful comments and  suggestions.

\printbibliography{}

@article {AABP24,
    AUTHOR = {Armana, C\'ecile and Angl\`es, Bruno and Bosser, Vincent and Pazuki, Fabien},
    shorthand = {AABP24},
     TITLE = {Drinfeld singular moduli, hyperbolas, units},
   JOURNAL = {Advances in Math.},
      YEAR = {to appear},
    EPRINT = {arXiv:2404.01075},
ARCHIVEPREFIX = {arXiv},
    NUMBER = {},
     PAGES = {},
      ISSN = {},
       DOI = {}
}

@article {Cornut02,
    AUTHOR = {Cornut, Christophe},
     TITLE = {Mazur's conjecture on higher {H}eegner points},
   JOURNAL = {Invent. Math.},
  FJOURNAL = {Inventiones Mathematicae},
    VOLUME = {148},
      YEAR = {2002},
    NUMBER = {3},
     PAGES = {495--523},
      ISSN = {0020-9910,1432-1297},
   MRCLASS = {11G40 (11F33 11F67 11G18 11R23)},
  MRNUMBER = {1908058},
MRREVIEWER = {Jan\ Nekov\'a\v r},
       DOI = {10.1007/s002220100199},
       URL = {https://doi.org/10.1007/s002220100199},
}

@article {Vatsal02,
    AUTHOR = {Vatsal, V.},
     TITLE = {Uniform distribution of {H}eegner points},
   JOURNAL = {Invent. Math.},
  FJOURNAL = {Inventiones Mathematicae},
    VOLUME = {148},
      YEAR = {2002},
    NUMBER = {1},
     PAGES = {1--46},
      ISSN = {0020-9910,1432-1297},
   MRCLASS = {11G40 (11G05 11R23)},
  MRNUMBER = {1892842},
MRREVIEWER = {Massimo\ Bertolini},
       DOI = {10.1007/s002220100183},
       URL = {https://doi.org/10.1007/s002220100183},
}

@article {JetchevKane,
    AUTHOR = {Jetchev, Dimitar and Kane, Ben},
     TITLE = {Equidistribution of {H}eegner points and ternary quadratic
              forms},
   JOURNAL = {Math. Ann.},
  FJOURNAL = {Mathematische Annalen},
    VOLUME = {350},
      YEAR = {2011},
    NUMBER = {3},
     PAGES = {501--532},
      ISSN = {0025-5831,1432-1807},
   MRCLASS = {11G18 (11E20 11E45 11G05)},
  MRNUMBER = {2805634},
MRREVIEWER = {\'Alvaro\ Lozano-Robledo},
       DOI = {10.1007/s00208-010-0568-5},
       URL = {https://doi-org.kb-ku.idm.oclc.org/10.1007/s00208-010-0568-5},
}

@article {CornutVatsal,
    AUTHOR = {Cornut, C. and Vatsal, V.},
     TITLE = {C{M} points and quaternion algebras},
   JOURNAL = {Doc. Math.},
  FJOURNAL = {Documenta Mathematica},
    VOLUME = {10},
      YEAR = {2005},
     PAGES = {263--309},
      ISSN = {1431-0635,1431-0643},
   MRCLASS = {11G15 (11G18)},
  MRNUMBER = {2148077},
MRREVIEWER = {Benjamin\ V.\ Howard},
}

@article {Michel04,
    AUTHOR = {Michel, P.},
     TITLE = {The subconvexity problem for {R}ankin-{S}elberg
              {$L$}-functions and equidistribution of {H}eegner points},
   JOURNAL = {Ann. of Math. (2)},
  FJOURNAL = {Annals of Mathematics. Second Series},
    VOLUME = {160},
      YEAR = {2004},
    NUMBER = {1},
     PAGES = {185--236},
      ISSN = {0003-486X,1939-8980},
   MRCLASS = {11F66 (11G18)},
  MRNUMBER = {2119720},
MRREVIEWER = {Gergely\ Harcos},
       DOI = {10.4007/annals.2004.160.185},
       URL = {https://doi-org.kb-ku.idm.oclc.org/10.4007/annals.2004.160.185},
}

@article {Diaconu,
    AUTHOR = {Diaconu, Adrian},
     TITLE = {On the third moment of {$L(\frac{1}{2},\chi_d)$} {I}: {T}he
              rational function field case},
   JOURNAL = {J. Number Theory},
  FJOURNAL = {Journal of Number Theory},
    VOLUME = {198},
      YEAR = {2019},
     PAGES = {1--42},
      ISSN = {0022-314X,1096-1658},
   MRCLASS = {11M06 (11R58)},
  MRNUMBER = {3912928},
MRREVIEWER = {Sandro\ Bettin},
       DOI = {10.1016/j.jnt.2018.09.023},
       URL = {https://doi.org/10.1016/j.jnt.2018.09.023},
}

@incollection {reversat,
    AUTHOR = {Reversat, Marc},
     TITLE = {Lecture on rigid geometry},
 BOOKTITLE = {The arithmetic of function fields ({C}olumbus, {OH}, 1991)},
    SERIES = {Ohio State Univ. Math. Res. Inst. Publ.},
    VOLUME = {2},
     PAGES = {143--151},
 PUBLISHER = {de Gruyter, Berlin},
      YEAR = {1992},
      ISBN = {3-11-013171-4},
   MRCLASS = {11G09 (32P05)},
  MRNUMBER = {1196517},
MRREVIEWER = {David\ Goss},
}

@incollection {gekeler99,
    AUTHOR = {Gekeler, Ernst-Ulrich},
     TITLE = {Some new results on modular forms for {${\rm GL}(2,{\bf
              F}_q[T])$}},
 BOOKTITLE = {Recent progress in algebra, 1997},
    SERIES = {Contemp. Math.},
    VOLUME = {224},
     PAGES = {111--141},
 PUBLISHER = {Amer. Math. Soc., Prov., RI},
      YEAR = {1999},
      ISBN = {0-8218-0972-5},
   MRCLASS = {11G09 (11F11 11R58)},
  MRNUMBER = {1653065},
MRREVIEWER = {Yoshinori\ Hamahata},
       DOI = {10.1090/conm/224/03189},
       URL = {https://doi.org/10.1090/conm/224/03189},
}

@article {gekelerreversat,
    AUTHOR = {Gekeler, E.-U. and Reversat, M.},
     TITLE = {Jacobians of {D}rinfeld modular curves},
   JOURNAL = {J. Reine Angew. Math.},
  FJOURNAL = {Journal f\"ur die Reine und Angewandte Mathematik. [Crelle's
              Journal]},
    VOLUME = {476},
      YEAR = {1996},
     PAGES = {27--93},
      ISSN = {0075-4102,1435-5345},
   MRCLASS = {11G09 (11R39)},
  MRNUMBER = {1401696},
MRREVIEWER = {David\ Goss},
       DOI = {10.1515/crll.1996.476.27},
       URL = {https://doi.org/10.1515/crll.1996.476.27},
}

@incollection {clozelullmo,
    AUTHOR = {Clozel, Laurent and Ullmo, Emmanuel},
     TITLE = {\'Equidistribution des points de {H}ecke},
 BOOKTITLE = {Contr. to automorphic forms, geometry, and number
              theory},
     PAGES = {193--254},
 PUBLISHER = {J. Hopkins Univ. Press, Baltimore, MD},
      YEAR = {2004},
      ISBN = {0-8018-7860-8},
   MRCLASS = {11F60 (11F70)},
  MRNUMBER = {2058609},
MRREVIEWER = {A.\ Raghuram},
}

@article {Goss80,
    AUTHOR = {Goss, David},
     TITLE = {{$\pi $}-adic {E}isenstein series for function fields},
   JOURNAL = {Comp. Math.},
  FJOURNAL = {Compositio Mathematica},
    VOLUME = {41},
      YEAR = {1980},
    NUMBER = {1},
     PAGES = {3--38},
      ISSN = {0010-437X,1570-5846},
   MRCLASS = {10D45 (10D99)},
  MRNUMBER = {578049},
MRREVIEWER = {K.\ Shiratani},
       URL = {http://www.numdam.org/item?id=CM_1980__41_1_3_0},
}

@article {Zhang,
    AUTHOR = {Zhang, Shouwu},
     TITLE = {Heights of {H}eegner points on {S}himura curves},
   JOURNAL = {Ann. of Math. (2)},
  FJOURNAL = {Annals of Mathematics. Second Series},
    VOLUME = {153},
      YEAR = {2001},
    NUMBER = {1},
     PAGES = {27--147},
      ISSN = {0003-486X,1939-8980},
   MRCLASS = {11G18 (11F11 11G05 11G40 11G50)},
  MRNUMBER = {1826411},
MRREVIEWER = {Henri\ Darmon},
       DOI = {10.2307/2661372},
       URL = {https://doi.org/10.2307/2661372},
}

@misc{papanikolasequidistribution,
      title={Equidistribution of {G}ross points over rational function fields}, 
      author={Ahmad El-Guindy and Riad Masri and Matthew Papanikolas and Guchao Zeng},
      shorthand={EGM\textsuperscript{+}20},
      year={2020},
      eprint={1905.07001},
      archivePrefix={arXiv},
      primaryClass={math.NT},
      url={https://arxiv.org/abs/1905.07001}, 
}

@book {rosen,
    AUTHOR = {Rosen, Michael},
     TITLE = {Number theory in function fields},
    SERIES = {Graduate Texts in Mathematics},
    VOLUME = {210},
 PUBLISHER = {Springer-Verlag, New York},
      YEAR = {2002},
     PAGES = {xii+358},
      ISBN = {0-387-95335-3},
   MRCLASS = {11R58 (11R60 11T55)},
  MRNUMBER = {1876657},
MRREVIEWER = {Ernst-Ulrich\ Gekeler},
       DOI = {10.1007/978-1-4757-6046-0},
       URL = {https://doi.org/10.1007/978-1-4757-6046-0},
}

@book {Bourbaki04,
    AUTHOR = {Bourbaki, Nicolas},
     TITLE = {Integration. {II}. {C}hapters 7--9},
    SERIES = {Elements of Mathematics (Berlin)},
      NOTE = {Translated from the 1963 and 1969 French originals by Sterling
              K. Berberian},
 PUBLISHER = {Springer-Verlag, Berlin},
      YEAR = {2004},
     PAGES = {viii+326},
      ISBN = {3-540-20585-3},
   MRCLASS = {28-01 (28C10 28C15)},
  MRNUMBER = {2098271},
}

@article{herreroMenaresRivera1,
  author={Herrero, Sebasti{\'a}n and Menares, Ricardo and Rivera-Letelier, Juan},
  title={p-adic distribution of CM points and Hecke orbits {I}: Convergence towards the Gauss point},
  journal={Algebra \& Number Theory},
  volume={14},
  number={5},
  pages={1239--1290},
  year={2020},
  publisher={Mathematical Sciences Publishers}
}

@misc{herreroMenaresRivera2,
    author={Herrero, Sebasti{\'a}n and Menares, Ricardo and Rivera-Letelier, Juan},
      title={p-Adic distribution of CM points and Hecke orbits. II: Linnik equidistribution on the supersingular locus},
      year={2021},
      eprint={2102.04865},
      archivePrefix={arXiv},
      primaryClass={math.NT},
      url={https://arxiv.org/abs/2102.04865}, 
}

@article{drinfeld74,
  title={Elliptic modules},
  author={Drinfel'd, Vladimir G},
  journal={Mathematics of the USSR-Sbornik},
  volume={23},
  number={4},
  pages={561},
  year={1974},
  publisher={IOP Publishing}
}

@book{papikianlibro,
  title={Drinfeld modules},
  author={Papikian, Mihran},
  volume={296},
  year={2023},
  publisher={Springer Nature}
}

@book {serretrees,
    AUTHOR = {Serre, Jean-Pierre},
     TITLE = {Trees},
      NOTE = {Translated from the French by John Stillwell},
 PUBLISHER = {Springer-Verlag, Berlin-New York},
      YEAR = {1980},
     PAGES = {ix+142},
      ISBN = {3-540-10103-9},
   MRCLASS = {20H10 (05C05 22E50)},
  MRNUMBER = {607504},
}

@article {dukeHiperbolic,
    AUTHOR = {Duke, W.},
     TITLE = {Hyperbolic distribution problems and half-integral weight
              {M}aass forms},
   JOURNAL = {Invent. Math.},
  FJOURNAL = {Inventiones Mathematicae},
    VOLUME = {92},
      YEAR = {1988},
    NUMBER = {1},
     PAGES = {73--90},
      ISSN = {0020-9910,1432-1297},
   MRCLASS = {11F11 (11E32 11F30 11F37)},
  MRNUMBER = {931205},
MRREVIEWER = {Mark\ Sheingorn},
       DOI = {10.1007/BF01393993},
       URL = {https://doi.org/10.1007/BF01393993},
}

@article{nagoshi,
  title={Selberg zeta functions over function fields},
  author={Nagoshi, Hirofumi},
  journal={Journal of Number Theory},
  volume={90},
  number={2},
  pages={207--238},
  year={2001},
  publisher={Elsevier}
}

@book{efrat,
  title={Automorphic spectra on the tree of {PGL}2},
  author={Efrat, Isaac},
  year={1986},
  publisher={Mathematical Sciences Research Institute}
}

@book{iwanieckowalski,
  title={Analytic number theory},
  author={Iwaniec, Henryk and Kowalski, Emmanuel},
  volume={53},
  year={2021},
  publisher={American Math. Soc.}
}

@inproceedings{bucur2018traces,
  title={Traces, high powers and one level density for families of curves over finite fields},
  author={Bucur, Alina and Costa, Edgar and David, Chantal and Guerreiro, Joao and Lowry--Duda, David},
  shorthand = {BCD\textsuperscript{+}18},
  booktitle={Mathematical Proceedings of the Cambridge Philosophical Society},
  volume={165},
  number={2},
  pages={225--248},
  year={2018},
  organization={Cambridge University Press}
}

\end{document}